\documentclass[oneside,reqno,10pt]{amsart}

\usepackage[margin=1in]{geometry}
\usepackage{mathtools}
\usepackage{amsmath}
\usepackage{amsthm}
\usepackage{amssymb}
\usepackage{tikz}
\usepackage{tikz-cd}
\usepackage[all,cmtip]{xy}
\usepackage{float}
\usepackage{indentfirst}
\usepackage[mathscr]{euscript}
\usepackage{stmaryrd}
\usepackage{hyperref}
\usepackage{csquotes}
\usepackage{enumitem}

\DeclareFontFamily{U}{mathx}{}
\DeclareFontShape{U}{mathx}{m}{n}{<-> mathx10}{}
\DeclareSymbolFont{mathx}{U}{mathx}{m}{n}
\DeclareMathAccent{\widehat}{0}{mathx}{"70}
\DeclareMathAccent{\widecheck}{0}{mathx}{"71}

\newcommand{\W}{\mathcal{W}}
\newcommand{\F}{\mathcal{F}}
\newcommand{\Hom}{\mathrm{Hom}}

\makeatletter
\def\blfootnote{\gdef\@thefnmark{}\@footnotetext}
\makeatother

\theoremstyle{plain}
\newtheorem{thm}{Theorem}[section]
\newtheorem{prop}[thm]{Proposition}
\newtheorem{lem}[thm]{Lemma}
\newtheorem{cor}[thm]{Corollary}

\theoremstyle{definition}
\newtheorem{dfn}[thm]{Definition}
\newtheorem{exa}[thm]{Example}
\newtheorem*{ack}{Acknowledgements}

\theoremstyle{remark}
\newtheorem{rmk}[thm]{Remark}

\numberwithin{equation}{section}

\title[Cluster categories from Fukaya categories]{Cluster categories from Fukaya categories}

\author[Bae]{Hanwool Bae}
\address{Center for Quantum Structures in Modules and Spaces, Seoul National University, Seoul 08826, Republic of Korea}
\email{hanwoolb@gmail.com}

\author[Jeong]{Wonbo Jeong}
\address{Department of Mathematical Sciences, Research Institute in Mathematics, Seoul National University, Seoul 08826, Republic of Korea}
\email{wonbo.jeong@gmail.com}

\author[Kim]{Jongmyeong Kim}
\address{Center for Geometry and Physics, Institute for Basic Science (IBS), Pohang 37673, Republic of Korea}
\email{myeong@ibs.re.kr}

\begin{document}

\begin{abstract}
We show that the derived wrapped Fukaya category $D^\pi\mathcal{W}(X_{Q}^{d+1})$, the derived compact Fukaya category $D^\pi\mathcal{F}(X_{Q}^{d+1})$ and the cocore disks $L_{Q}$ of the plumbing space $X_{Q}^{d+1}$ form a Calabi--Yau triple. As a consequence, the quotient category $D^\pi\mathcal{W}(X_{Q}^{d+1})/D^\pi\mathcal{F}(X_{Q}^{d+1})$ becomes the cluster category associated to $Q$. One of its properties is a Calabi--Yau structure. Also it is known that this quotient category is quasi-equivalent to the Rabinowitz Fukaya category due to the work of Ganatra--Gao--Venkatesh. We compute the morphism space of $L_{Q}$ in $D^\pi\mathcal{W}(X_{Q}^{d+1})/D^\pi\mathcal{F}(X_{Q}^{d+1})$ using the Calabi--Yau structure, which is isomorphic to the Rabinowitz Floer cohomology of $L_{Q}$.
\end{abstract}

\maketitle

\tableofcontents

\blfootnote{\textit{2020 Mathematics Subject Classification}. Primary 53D37; Secondary 16E45, 18G80 \\
\indent\textit{Key Words and Phrases}. cluster categories, Calabi--Yau triples, Fukaya categories }

\section{Introduction}

A cluster category was introduced as a categorification of the cluster algebra. It can be used to investigate the corresponding cluster algebra using the categorical language. But the cluster structure itself is an interesting categorical structure on the triangulated category. It has been studied intensively since its appearance (see \cite{bmrrt06}, \cite{ami09}, \cite{iya-yan18} and references therein).

In symplectic geometry, there are many triangulated categories for a given symplectic manifold. Let $M$ be a {\it good} Liouville domain. Then we have the derived wrapped Fukaya category $D^{\pi}\mathcal{W}(M)$ and the derived compact Fukaya category $D^{\pi}\mathcal{F}(M)$, which are triangulated categories. By definition, $D^{\pi}\mathcal{F}(M)$ is a subcategory of $D^{\pi}\mathcal{W}(M)$ and we can take the quotient $D^{\pi}\mathcal{W}(M)/D^{\pi}\mathcal{F}(M)$, which is also a triangulated category.

In this paper, we consider the plumbing space $X_{Q}^{d+1}$ of $T^{*}S^{d+1}$ ($d \geq 2$) along a quiver $Q$ whose underlying  graph is a tree $T$. This space was investigated in \cite{ekh-lek17} and was shown to have many nice properties, e.g., Koszul duality. We will show that the quotient category $D^{\pi}\mathcal{W}(X_{Q}^{d+1})/D^{\pi}\mathcal{F}(X_{Q}^{d+1})$ becomes a cluster category in the sense of \cite{iya-yan18}.

The following are the main lemma and theorem of this paper.

\begin{lem}
Let $\mathcal{C}$ be a split-closed algebraic triangulated category and $\mathcal{D} \subset \mathcal{C}$ be a thick subcategory.
Let $M \cong M_1 \oplus \cdots \oplus M_n \in \mathcal{C}$ be a basic silting object (a generator with non-positively supported endomorphism space) where each $M_i$ is indecomposable.
Suppose it satisfies the following conditions:
\begin{enumerate}
\item $\mathrm{Hom}_\mathcal{C}(M,M)$ is finite dimensional.
\item There is a collection $\{S_1,\dots,S_n\} \subset \mathcal{D}$ of objects which split-generates $\mathcal{D}$ such that
\begin{equation*}
\mathrm{Hom}_\mathcal{C}(M_i,S_j[p]) \cong
\begin{cases}
\mathbb{K} & (i=j \text{ and } p=0),\\
0 & (\text{otherwise}).
\end{cases}
\end{equation*}
\end{enumerate}
Then $(\mathcal{C},\mathcal{D},M)$ is an ST-triple.
If moreover $(\mathcal{C},\mathcal{D})$ is relative $(d+1)$-Calabi--Yau then $(\mathcal{C},\mathcal{D},M)$ is a $(d+1)$-Calabi--Yau triple.
\end{lem}

\begin{thm} \label{thm:CYtriple}
$(D^{\pi}\mathcal{W}(X_{Q}^{d+1}),D^{\pi}\mathcal{F}(X_{Q}^{d+1}),L_{Q})$ is a $(d+1)$-Calabi--Yau triple where $L_{Q} = \oplus_{i \in Q_{0}} L_{i}$ is the direct sum of cocore disks (cotangent fibers of each $T^{*}S^{d+1}$) in $X_{Q}^{d+1}$. Therefore, $D^{\pi}\mathcal{W}(X_{Q}^{d+1})/D^{\pi}\mathcal{F}(X_{Q}^{d+1})$ becomes the cluster category associated to $Q$.
\end{thm}

\begin{rmk}
A result similar to the above theorem appeared in \cite{lek-ued18} for the first time. They conjectured the same for some Milnor fibers and proved it in some cases using the homological mirror symmetry. We take a more direct approach to prove our main theorem. Indeed, we do not use the homological mirror symmetry, but only use the categorical data in the Fukaya categories.
\end{rmk}

Theorem \ref{thm:CYtriple} gives many categorical structures on the quotient category $D^{\pi}\mathcal{W}(X_{Q}^{d+1})/D^{\pi}\mathcal{F}(X_{Q}^{d+1})$. For example, $D^{\pi}\mathcal{W}(X_{Q}^{d+1})/D^{\pi}\mathcal{F}(X_{Q}^{d+1})$ has the $d$-Calabi--Yau structure. If we assume that the work of Ganatra--Gao--Venkatesh \cite{GGV}, this quotient category is equivalent to the derived Rabinowitz Fukaya category. Thus,  we have the following corollary.

\begin{cor}
The derived Rabinowitz Fukaya category $D^{\pi}\mathcal{RW}(X_{Q}^{d+1})$ of $X_{Q}^{d+1}$ is the cluster category associated to $Q$. In particular, $D^{\pi}\mathcal{RW}(X_{Q}^{d+1})$ is a $d$-Calabi--Yau triangulated category.
\end{cor}

For a geometric interpretation of the above result, we remark that the Calabi--Yau structure on $D^{\pi} \mathcal{R}\mathcal{W}(X_Q^{d+1})$ can be understood as a categorical extension of the Poincar\'{e} duality \cite{cho20}. See \cite{BJK23} for details.
We concentrate only on the Calabi--Yau structures on the cluster categories in this paper. However, for cluster mutations, we expect that the mutations are related to Lagrangian surgeries and leave it for future research.

%Motivated by this result, we reconstruct a natural non-degenerate pairing on the Rabinowitz Floer cohomology in a symplecto-geometric way. More precisely, we consider more general setting first (see Section \ref{section:symplectic} for details). Let $M$ be a good Liouville domain of dimension $2(d+1)$. We construct an $A_{3}$-structure on the Rabinowitz Floer complexes of admissible Lagrangians. Especially, we define the product structure
%$$\widecheck{\mu}^{2} : RFC^{*}(K_{1},K_{2}) \otimes RFC^{*}(K_{0},K_{1}) \to RFC^{*}(K_{0},K_{2})$$
%using popsicle maps introduced in \cite{abo-sei10} and show that it induces an associative product on the cohomology. Then the pairing is given by the composition of this product and the integration. It can be shown that this pairing can be extended to the triangulated envelope. As a result, this gives a symplecto-geometric interpretation of $d$-Calabi--Yau structure on $D^{\pi}\mathcal{RW}(M)$.
%
%\begin{thm}
%For a collection of admissible Lagrangians $\{ K_{i} \}$, there exist $A_{3}$-structure on the Rabinowitz Floer complex $RFC^{*}(K_{i},K_{j})$ such that a differential $\check{\mu}^{1}$, $\check{\mu}^{2}$ and $\check{\mu}^{3}$ satisfy the $A_{3}$-relation. Moreover, there exists a natural non-degenerate pairing
%$$\beta : RFH^{d-*}(K_{j},K_{i}) \otimes RFH^{*}(K_{i},K_{j}) \to \mathbb{K}$$
%defined by the composition of product and integration. This pairing can be extended to $D^{\pi}\mathcal{RW}(M)$ under some assumptions.
%\end{thm}
%
%Finally, we return to the plumbing case.

Using the Calabi--Yau structure, we compute the morphism spaces in $D^{\pi}\mathcal{W}(X_{Q}^{3})/D^{\pi}\mathcal{F}(X_{Q}^{3})$ between the cocore disks $L_{i}$ when the quiver is of Dynkin type and $d=2$. One can regard the following result as the computation of Rabinowitz Floer cohomology (or v-shaped Lagrangian Floer cohomology) between the cocores. Then it becomes easy to compute the pairing for the silting object $L_{Q}$ given as the direct sum of all the cocore disks.

\begin{thm}
Given a Dynkin quiver $Q$, one can construct a quiver $\overline{\Omega}_{Q}$ (see Definition \ref{dfn:newomega}) and an ideal $J$ in $\mathbb{K}\overline{\Omega}_{Q}$. Then, there is a graded ring isomorphism
$$\phi : \mathbb{K}\overline{\Omega}_{Q}/\overline{J} \xrightarrow{\hspace{0.3em} \simeq \hspace{0.3em}} \mathrm{Hom}^{*}_{D^{\pi}\mathcal{W}(X_{Q}^{3})/D^{\pi}\mathcal{F}(X_{Q}^{3})}(L_{Q},L_{Q})$$
(which can be generalized for any $d \geq 2$).
\end{thm}

%Note that we concentrate only on the Calabi--Yau structures on the cluster categories in this paper. However, there are other interesting structures such as a cluster mutation. We expect that the mutations are related to Lagrangian surgeries and leave it for the further research.

\subsection{Notations and conventions}

In this paper, $\mathbb{K}$ denotes an algebraically closed field of characteristic zero.

All triangulated categories and exact functors are assumed to be linear over $\mathbb{K}$.
The space of morphisms from two objects $X,Y$ of a triangulated category $\mathcal{C}$ will be denoted by $\mathrm{Hom}_\mathcal{C}(X,Y)$.
We also use the notation
\begin{equation*}
\mathrm{Hom}^*_\mathcal{C}(X,Y) \coloneqq \bigoplus_{p \in \mathbb{Z}} \mathrm{Hom}_\mathcal{C}(X,Y[p])[-p].
\end{equation*}

Similarly, we assume that all $A_\infty$-categories and $A_\infty$-functors are linear over $\mathbb{K}$ and cohomologically unital in the sense of \cite[Section (2a)]{sei08}.
The space of morphisms from two objects $X,Y$ of an $A_\infty$-category $\mathcal{A}$ will be denoted by $\mathrm{hom}_\mathcal{A}(X,Y)$.

\begin{ack}
We thank Yuan Gao for helpful discussion about their new results.
Hanwool Bae was supported by the National Research Foundation of Korea (NRF) grant funded by the Korea government (MSIT) (No.2020R1A5A1016126), Wonbo Jeong was supported by Samsung Science and Technology Foundation under project number SSTF-BA1402-52, and Jongmyeong Kim was supported by the Institute for Basic Science (IBS-R003-D1).
\end{ack}

\section{$A_\infty$-categories and $A_\infty$-modules}

\subsection{$A_\infty$-categories}

In this section, we review basic properties of $A_\infty$-modules and introduce some notations.
For the precise definitions, see \cite[Section 2]{gan13}.

An {\em $A_\infty$-category} $\mathcal{C}$ consists of a set of objects $\mathrm{Ob}(\mathcal{C})$, a graded vector space $\mathrm{hom}_\mathcal{C}(X,Y)$ for every $X,Y \in \mathrm{Ob}(\mathcal{C})$, and graded linear maps
\begin{equation*}
\mu_\mathcal{C}^k : \mathrm{hom}_\mathcal{C}(X_{k-1},X_k) \otimes \cdots \otimes \mathrm{hom}_\mathcal{C}(X_0,X_1) \to \mathrm{hom}_\mathcal{C}(X_0,X_k)[2-k]
\end{equation*}
for all $k \geq 1$ and $X_0,\dots,X_k \in \mathrm{Ob}(\mathcal{C})$ satisfying the $A_\infty$-relations
\begin{equation*}
\sum_{n,m} (-1)^{*_n} \mu_\mathcal{C}^{k-m+1}(a_k,\dots,a_{n+m+1},\mu_\mathcal{C}^m(a_{n+m},\dots,a_{n+1}),a_n,\dots,a_1) = 0
\end{equation*}
where $*_n = \deg a_{1} + \dots + \deg a_{n} - n$.
An $A_\infty$-category $\mathcal{C}$ is called {\em cohomologically unital} if the homology category $H^*(\mathcal{C})$ has identity morphisms for all objects (i.e., $H^*(\mathcal{C})$ is a category in the ordinary sense).

An {\em $A_\infty$-functor} $F : \mathcal{C} \to \mathcal{D}$ between $A_\infty$-categories $\mathcal{C},\mathcal{D}$ consists of a map $F : \mathrm{Ob}(\mathcal{C}) \to \mathrm{Ob}(\mathcal{D})$, and graded linear maps
\begin{equation*}
F^k : \mathrm{hom}_\mathcal{C}(X_{k-1},X_k) \otimes \cdots \otimes \mathrm{hom}_\mathcal{C}(X_0,X_1) \to \mathrm{hom}_\mathcal{D}(FX_0,FX_k)[1-k]
\end{equation*}
for any $k \geq 1$ and $X_0,\dots,X_k \in \mathrm{Ob}(\mathcal{C})$ satisfying the $A_\infty$-relations
\begin{align*}
&\sum_{n,m} (-1)^{*_n} F^{k-m+1}(a_k,\dots,a_{n+m+1},\mu_\mathcal{C}^m(a_{n+m},\dots,a_{n+1}),a_n,\dots,a_1)\\
&= \sum_r \sum_{s_1,\dots,s_r} \mu_\mathcal{D}^r(F^{s_r}(a_k,\dots,a_{k-s_r+1}),\dots,F^{s_1}(a_{s_1},\dots,s_1)).
\end{align*}
An $A_\infty$-functor $F : \mathcal{C} \to \mathcal{D}$ between cohomologically unital $A_\infty$-categories $\mathcal{C},\mathcal{D}$ is called {\em cohomologically unital} if the induced functor $H^*(F) : H^*(\mathcal{C}) \to H^*(\mathcal{D})$ sends identity morphisms to identity morphisms (i.e., $H^*(F)$ is a functor in the ordinary sense).

In this paper, we only consider cohomologically unital $A_\infty$-categories and $A_\infty$-functors.

\subsection{$A_\infty$-modules}

Let $\mathcal{A}$ be an $A_\infty$-category.
Recall that the dg category $\mathrm{Mod}_\mathcal{A}$ of right {\em $\mathcal{A}$-modules} is defined to be the dg category $\mathrm{Fun}_{A_\infty}(\mathcal{A}^\mathrm{op},\mathrm{Mod}_\mathbb{K})$ of $A_\infty$-functors from $\mathcal{A}^\mathrm{op}$ to the dg category $\mathrm{Mod}_\mathbb{K}$ of dg $\mathbb{K}$-modules (see \cite[Section 2.4]{gan13}).
Similarly, the dg category $_\mathcal{A}\mathrm{Mod}$ of left $\mathcal{A}$-modules is defined to be $\mathrm{Fun}_{A_\infty}(\mathcal{A},\mathrm{Mod}_\mathbb{K})$.
For two $A_\infty$-categories $\mathcal{A},\mathcal{B}$, we denote by $_\mathcal{A}\mathrm{Mod}_\mathcal{B}$ the dg category of $(\mathcal{A},\mathcal{B})$-bimodules.
If $\mathcal{A}$ (resp. $\mathcal{B}$) is $\mathbb{K}$ then $_\mathcal{A}\mathrm{Mod}_\mathcal{B}$ can naturally be identified with $\mathrm{Mod}_\mathcal{B}$ (resp. $_\mathcal{A}\mathrm{Mod}$).
For convenience, we often denote $\mathrm{hom}_{\mathrm{Mod}_\mathcal{A}}$, $\mathrm{hom}_{_\mathcal{A}\mathrm{Mod}}$ and $\mathrm{hom}_{_\mathcal{A}\mathrm{Mod}_\mathcal{B}}$ by $\mathrm{hom}_\mathcal{A}$, $_\mathcal{A}\mathrm{hom}$ and $_\mathcal{A}\mathrm{hom}_\mathcal{B}$ respectively.

For $A_\infty$-categories $\mathcal{A},\mathcal{B},\mathcal{C}$ and $\mathcal{M} \in {_\mathcal{A}}\mathrm{Mod}_\mathcal{B}$, there is a dg functor
\begin{equation*}
- \otimes_\mathcal{A} \mathcal{M} : {_\mathcal{C}}\mathrm{Mod}_\mathcal{A} \to {_\mathcal{C}}\mathrm{Mod}_\mathcal{B}
\end{equation*}
which generalizes the tensor product of dg modules (see \cite[Section 2.5]{gan13}).
Note that if $\mathcal{M} \in {_\mathcal{A}}\mathrm{Mod}$ and $\mathcal{N} \in \mathrm{Mod}_\mathcal{B}$ then, regarding them as $(\mathcal{A},\mathbb{K})$- and $(\mathbb{K},\mathcal{B})$-bimodules respectively, we obtain $\mathcal{M} \otimes_\mathbb{K} \mathcal{N} \in {_\mathcal{A}}\mathrm{Mod}_\mathcal{B}$.
There is also a notion of the bimodule tensor product
\begin{equation*}
- \otimes_{(\mathcal{A},\mathcal{B})} \mathcal{M} : {_\mathcal{B}}\mathrm{Mod}_\mathcal{A} \to \mathrm{Mod}_\mathbb{K}.
\end{equation*}

It is not difficult to check the following hom-tensor adjunction for $A_\infty$-modules (see \cite[Proposition 2.10]{gan13} in which it is proved when $\mathcal{A}=\mathcal{C}$ and $\mathcal{B}=\mathbb{K}$).

\begin{prop}\label{hom-ten}
Let $\mathcal{M} \in {_\mathcal{A}}\mathrm{Mod}_\mathcal{B}$, $\mathcal{N} \in {_\mathcal{B}}\mathrm{Mod}_\mathcal{C}$ and $\mathcal{L} \in {_\mathcal{A}}\mathrm{Mod}_\mathcal{C}$.
Then we have natural quasi-isomorphisms
\begin{gather*}
{_\mathcal{A}}\mathrm{hom}_\mathcal{C}(\mathcal{M} \otimes_\mathcal{B} \mathcal{N},\mathcal{L}) \simeq {_\mathcal{A}}\mathrm{hom}_\mathcal{B}(\mathcal{M},\mathrm{hom}_\mathcal{C}(\mathcal{N},\mathcal{L})),\\
{_\mathcal{A}}\mathrm{hom}_\mathcal{C}(\mathcal{M} \otimes_\mathcal{B} \mathcal{N},\mathcal{L}) \simeq {_\mathcal{B}}\mathrm{hom}_\mathcal{C}(\mathcal{N},{_\mathcal{A}}\mathrm{hom}(\mathcal{M},\mathcal{L})).
\end{gather*}
\end{prop}

There is a cohomologically fully faithful $A_\infty$-functor $\mathcal{Y}^r : \mathcal{A} \to \mathrm{Mod}_\mathcal{A}$, called the {\em right Yoneda embedding}, which sends $X \in \mathcal{A}$ to an $A_\infty$-functor $\mathcal{Y}^r_X : \mathcal{A}^\mathrm{op} \to \mathrm{Mod}_\mathbb{K}$ acting on objects by $\mathcal{Y}^r_X(Y) = \mathrm{hom}_\mathcal{A}(Y,X)$ (see \cite[Section 2.7]{gan13}).
Similarly, we also have the {\em left Yoneda embedding} $\mathcal{Y}^l : \mathcal{A}^\mathrm{op} \to {_\mathcal{A}\mathrm{Mod}}$.
A right (resp. left) $\mathcal{A}$-module of the form $\mathcal{Y}^r_X$ (resp. $\mathcal{Y}^l_X$) is called a {\em right (resp. left) Yoneda $\mathcal{A}$-module} and an $(\mathcal{A},\mathcal{B})$-bimodule of the form $\mathcal{Y}^l_X \otimes_\mathbb{K} \mathcal{Y}^r_Y$, for some $X \in \mathcal{A}$ and $Y \in \mathcal{B}$, is called a {\em Yoneda $(\mathcal{A},\mathcal{B})$-bimodule}.

The dg category $\mathrm{Mod}_\mathcal{A}$ has a natural structure of a pretriangulated dg category.
Let $\mathcal{X} \subset \mathcal{A}$ be a full subcategory and $\mathcal{M} \in \mathrm{Mod}_\mathcal{A}$.
We say that $\mathcal{M}$ is {\em split-generated} by $\mathcal{X}$ (or $\mathcal{X}$ {\em split-generates} $\mathcal{M}$) if $\mathcal{M}$ is contained in the smallest split-closed pretriangulated subcategory containing $\{ \mathcal{Y}^r_X \,|\, X \in \mathcal{X} \}$.
In this case, we also say that $\mathcal{M}$ is split-generated by $\{ \mathcal{Y}^r_X \,|\, X \in \mathcal{X} \}$.
Similar definitions can be made for the cases of $_\mathcal{A}\mathrm{Mod}$ and $_\mathcal{A}\mathrm{Mod}_\mathcal{B}$.

A right $\mathcal{A}$-module which can be split-generated by right Yoneda $\mathcal{A}$-modules is said to be {\em perfect}.
We denote by $\mathrm{Perf}_\mathcal{A}$ the full subcategory of $\mathrm{Mod}_\mathcal{A}$ consisting of perfect right $\mathcal{A}$-modules.
One can similarly define $_\mathcal{A}\mathrm{Perf}$ and $_\mathcal{A}\mathrm{Perf}_\mathcal{B}$.

We also have a full subcategory of $\mathrm{Mod}_\mathcal{A} (= \mathrm{Fun}_{A_\infty}(\mathcal{A}^\mathrm{op},\mathrm{Mod}_\mathbb{K}))$ consisting of $A_\infty$-functors from $\mathcal{A}^\mathrm{op}$ to $\mathrm{Mod}_\mathbb{K}$ whose images lie in the dg category $\mathrm{Perf}_\mathbb{K}$ of perfect dg $\mathbb{K}$-modules, i.e., cochain complexes of $\mathbb{K}$-vector spaces with finite dimensional total cohomology.
A right $\mathcal{A}$-module in this subcategory is said to be {\em proper}.
We denote by $\mathrm{Prop}_\mathcal{A}$ the full subcategory of $\mathrm{Mod}_\mathcal{A}$ consisting of proper right $\mathcal{A}$-modules.
One can similarly define $_\mathcal{A}\mathrm{Prop}$ and $_\mathcal{A}\mathrm{Prop}_\mathcal{B}$.

The derived category $D\mathcal{A}$ of $\mathcal{A}$ is defined as the localization of $H^0(\mathrm{Mod}_\mathcal{A})$ with respect to $A_\infty$-quasi-isomorphisms.
It admits a natural structure of a triangulated category.
Moreover it has two distinguished thick subcategories: one is $D^\pi\mathcal{A}$ whose objects are perfect $\mathcal{A}$-modules and the other is $D_c\mathcal{A}$ whose objects are proper $\mathcal{A}$-modules.

\begin{rmk}
This notation is not standard.
In many literatures, especially when $\mathcal{A}$ is just a dg algebra (i.e., a dg category with a single object), $D^\pi\mathcal{A}$ and $D_c\mathcal{A}$ are usually denoted by $\mathrm{per}\,\mathcal{A}$ and $D_{fd}\mathcal{A}$ respectively.
\end{rmk}

\begin{rmk}
The main categories of study in this paper are $D^\pi\mathcal{W}(M)$ and $D_c\mathcal{W}(M)$ where $\mathcal{W}(M)$ is the wrapped Fukaya category of a Liouville manifold $M$.
It is natural to expect that $D_c\mathcal{W}(M)$ coincides with the derived compact Fukaya category $D^\pi\mathcal{F}(M)$ and, in fact, we will see this is the case under some assumptions (see Proposition \ref{prop:cytriple}).
However note that this is not always true as can be seen from the case $M=T^*S^1$ in which case $D^\pi\mathcal{F}(M)$ is strictly smaller than $D_c\mathcal{W}(M)$.
We also note that a similar problem was studied in \cite{gan21}.
\end{rmk}

\subsection{Smooth $A_\infty$-categories}

Recall that, for an $A_\infty$-category $\mathcal{A}$, the {\em diagonal bimodule} $\mathcal{A}_\Delta \in {_\mathcal{A}}\mathrm{Mod}_\mathcal{A}$ is defined by $\mathcal{A}_\Delta(X,Y) = \mathrm{hom}_\mathcal{A}(Y,X)$ equipped with an obvious $A_\infty$-structure.
Note that the dg functor $- \otimes_\mathcal{A} \mathcal{A}_\Delta : {_\mathcal{B}}\mathrm{Mod} \to {_\mathcal{B}}\mathrm{Mod}$ is naturally quasi-isomorphic to the identity functor.

\begin{dfn}[{\cite[Definition 8.2]{kon-soi09}}]
An $A_\infty$-category $\mathcal{A}$ is called {\em smooth} if the diagonal bimodule $\mathcal{A}_\Delta$ is perfect.
\end{dfn}

\begin{rmk}
It is known that the dg category of perfect complexes on a separated scheme $X$ of finite type is smooth if and only if $X$ is smooth (see \cite[Proposition 3.13]{lun10}, \cite[Proposition 3.31]{orl16}).
\end{rmk}

For $\mathcal{M} \in \mathrm{Mod}_\mathcal{A}$ (resp. $\mathcal{M} \in {_\mathcal{A}}\mathrm{Mod}$), we define its {\em linear dual} by $D\mathcal{M} \coloneqq {_\mathbb{K}}\mathrm{hom}(\mathcal{M},\mathbb{K}) \in {_\mathcal{A}}\mathrm{Mod}$ (resp. $D\mathcal{M} \coloneqq \mathrm{hom}_\mathbb{K}(\mathcal{M},\mathbb{K}) \in \mathrm{Mod}_\mathcal{A}$).
Also, for $\mathcal{M} \in {_\mathcal{A}}\mathrm{Mod}_\mathcal{A}$, one can define the {\em bimodule dual} $\mathcal{M}^! \in {_\mathcal{A}}\mathrm{Mod}_\mathcal{A}$ which acts on objects as $\mathcal{M}^!(X,Y) = {_\mathcal{A}}\mathrm{hom}_\mathcal{A}(\mathcal{M},\mathcal{Y}^l_X \otimes_\mathbb{K} \mathcal{Y}^r_Y)$.
See \cite[Definition 2.40]{gan13} for the precise definition.

The linear dual shares many properties with the linear dual for vector spaces.
For instance, one can easily check the following lemmas.

\begin{lem}\label{dual-dual}
For $\mathcal{M} \in \mathrm{Prop}_\mathcal{A}$, we have a natural quasi-isomorphism
\begin{equation*}
DD\mathcal{M} \simeq \mathcal{M}.
\end{equation*}
\end{lem}

\begin{lem}\label{lin-dual}
For $\mathcal{M} \in \mathrm{Prop}_\mathcal{A}$ and $\mathcal{N} \in \mathrm{Mod}_\mathcal{B}$, we have a natural quasi-isomorphism
\begin{equation*}
D\mathcal{M} \otimes_\mathbb{K} \mathcal{N} \simeq {_\mathbb{K}}\mathrm{hom}(\mathcal{M},\mathcal{N}).
\end{equation*}
\end{lem}

The following can be thought of as a bimodule version of Lemma \ref{lin-dual}.

\begin{prop}[{\cite[Proposition 2.16]{gan13}}]\label{bimod-dual}
Let $\mathcal{M},\mathcal{N} \in {_\mathcal{A}}\mathrm{Perf}_\mathcal{A}$.
Then we have a natural quasi-isomorphism
\begin{equation*}
\mathcal{M}^! \otimes_{(\mathcal{A},\mathcal{A})} \mathcal{N} \simeq {_\mathcal{A}}\mathrm{hom}_\mathcal{A}(\mathcal{M},\mathcal{N}).
\end{equation*}
\end{prop}

The following can be proved similarly as \cite[Lemma 4.1]{kel08} (also see \cite[Lemma 3.4]{kel11}).

\begin{prop}\label{sm-prop}
Let $\mathcal{A}$ be a smooth $A_\infty$-category.
Then $\mathrm{Prop}_\mathcal{A} \subset \mathrm{Perf}_\mathcal{A}$ and, for any $\mathcal{M} \in \mathrm{Prop}_\mathcal{A}$ and $\mathcal{N} \in \mathrm{Perf}_\mathcal{A}$, there is a natural quasi-isomorphism
\begin{equation*}
D\mathrm{hom}_\mathcal{A}(\mathcal{M},\mathcal{N}) \simeq \mathrm{hom}_\mathcal{A}(\mathcal{N} \otimes_\mathcal{A} \mathcal{A}_\Delta^!,\mathcal{M}).
\end{equation*}
\end{prop}

\begin{rmk}
Note that if $\mathcal{A}$ is a right proper $A_\infty$-category, i.e., every right Yoneda $\mathcal{A}$-module is proper, then clearly we have $\mathrm{Prop}_\mathcal{A} \supset \mathrm{Perf}_\mathcal{A}$.
\end{rmk}

\begin{proof}
Let us first show that $\mathrm{Prop}_\mathcal{A} \subset \mathrm{Perf}_\mathcal{A}$.
Let $\mathcal{M} \in \mathrm{Prop}_\mathcal{A}$.
Since $\mathcal{M} \simeq \mathcal{M} \otimes_\mathcal{A} \mathcal{A}_\Delta$ and $\mathcal{A}_\Delta \in {_\mathcal{A}}\mathrm{Perf}_\mathcal{A}$ by assumption, it is enough to show that
\begin{equation*}
\mathcal{M} \otimes_\mathcal{A} \mathcal{P} \in \mathrm{Perf}_\mathcal{A}
\end{equation*}
for every $\mathcal{P} \in {_\mathcal{A}}\mathrm{Perf}_\mathcal{A}$.
Indeed, as $\mathcal{P}$ is split-generated by Yoneda $(\mathcal{A},\mathcal{A})$-bimodules $\mathcal{Y}^l_X \otimes_\mathbb{K} \mathcal{Y}^r_Y$, we see that $\mathcal{M} \otimes_\mathcal{A} \mathcal{P}$ is split-generated by right $\mathcal{A}$-modules of the form $\mathcal{M} \otimes_\mathcal{A} \mathcal{Y}^l_X \otimes_\mathbb{K} \mathcal{Y}^r_Y \simeq \mathcal{M}(X) \otimes_\mathbb{K} \mathcal{Y}^r_Y$.
Since $\mathcal{M}(X) \in \mathrm{Perf}_\mathbb{K}$, each $\mathcal{M}(X) \otimes_\mathbb{K} \mathcal{Y}^r_Y$ is split-generated by $\mathcal{Y}^r_Y$.
This shows that $\mathcal{M} \otimes_\mathcal{A} \mathcal{P}$ is split-generated by right Yoneda $\mathcal{A}$-modules.

Now let $\mathcal{M} \in \mathrm{Prop}_\mathcal{A}$ and $\mathcal{N} \in \mathrm{Perf}_\mathcal{A}$.
Since $D\mathcal{M} \in {_\mathcal{A}}\mathrm{Prop}$, we see that $D\mathcal{M} \in {_\mathcal{A}}\mathrm{Perf}$ as above.
This shows that $D\mathcal{M} \otimes_\mathbb{K} \mathcal{N} \in {_\mathcal{A}}\mathrm{Perf}_\mathcal{A}$.
Then we have natural quasi-isomorphisms
\begin{align*}
\mathcal{N} \otimes_\mathcal{A} \mathcal{A}_\Delta^! \otimes_\mathcal{A} D\mathcal{M}
&\simeq (D\mathcal{M} \otimes_\mathbb{K} \mathcal{N}) \otimes_{(\mathcal{A},\mathcal{A})} \mathcal{A}_\Delta^!\\
&\simeq {_\mathcal{A}}\mathrm{hom}_\mathcal{A}(\mathcal{A}_\Delta,D\mathcal{M} \otimes_\mathbb{K} \mathcal{N})\\
&\simeq {_\mathcal{A}}\mathrm{hom}_\mathcal{A}(\mathcal{A}_\Delta,{_\mathbb{K}}\mathrm{hom}(\mathcal{M},\mathcal{N}))\\
&\simeq \mathrm{hom}_\mathcal{A}(\mathcal{M} \otimes_\mathcal{A} \mathcal{A}_\Delta,\mathcal{N})\\
&\simeq \mathrm{hom}_\mathcal{A}(\mathcal{M},\mathcal{N})
\end{align*}
where the second line follows from Proposition \ref{bimod-dual}, the third line follows from Lemma \ref{lin-dual} and the fourth line follows from Proposition \ref{hom-ten}.
Therefore we have
\begin{align*}
D\mathrm{hom}_\mathcal{A}(\mathcal{M},\mathcal{N})
&\simeq \mathrm{hom}_\mathbb{K}(\mathcal{N} \otimes_\mathcal{A} \mathcal{A}_\Delta^! \otimes_\mathcal{A} D\mathcal{M},\mathbb{K})\\
&\simeq \mathrm{hom}_\mathcal{A}(\mathcal{N} \otimes_\mathcal{A} \mathcal{A}_\Delta^!,\mathrm{hom}_\mathbb{K}(D\mathcal{M},\mathbb{K}))\\
&\simeq \mathrm{hom}_\mathcal{A}(\mathcal{N} \otimes_\mathcal{A} \mathcal{A}_\Delta^!,\mathcal{M})
\end{align*}
where the second line follows from Proposition \ref{hom-ten} and the third line follows from Lemma \ref{dual-dual}.
\end{proof}

\subsection{Calabi--Yau $A_\infty$-categories}

A triangulated category $\mathcal{C}$ is called {\em $d$-Calabi--Yau} if $\mathrm{Hom}_\mathcal{C}(X,Y)$ is finite dimensional for every $X,Y \in \mathcal{C}$ and there is a natural isomorphism
\begin{equation*}
D\mathrm{Hom}_\mathcal{C}(X,Y) \cong \mathrm{Hom}_\mathcal{C}(Y,X[d])
\end{equation*}
for every $X,Y \in \mathcal{C}$, or equivalently, if there is a non-degenerate pairing
\begin{equation*}
\beta_{X,Y} : \mathrm{Hom}_\mathcal{C}(Y,X[d]) \times \mathrm{Hom}_\mathcal{C}(X,Y) \to \mathbb{K}
\end{equation*}
which is bifunctorial in $X,Y \in \mathcal{C}$.
By abuse of notation, we also regard $\beta_{X,Y}$ as a linear map $\mathrm{Hom}_\mathcal{C}(Y,X[d]) \otimes \mathrm{Hom}_\mathcal{C}(X,Y) \to \mathbb{K}$.
We will sometimes refer to $\beta$ as the {\em $d$-Calabi--Yau pairing} on $\mathcal{C}$.

A pair $(\mathcal{C},\mathcal{D})$ of a triangulated category $\mathcal{C}$ and a thick subcategory $\mathcal{D} \subset \mathcal{C}$ is called {\em relative $(d+1)$-Calabi--Yau} if $\mathrm{Hom}_\mathcal{C}(X,Y)$ is finite dimensional for every $X \in \mathcal{D},Y \in \mathcal{C}$ and there is a natural isomorphism
\begin{equation*}
D\mathrm{Hom}_\mathcal{C}(X,Y) \cong \mathrm{Hom}_\mathcal{C}(Y,X[d+1])
\end{equation*}
for every $X \in \mathcal{D}$ and $Y \in \mathcal{C}$, or equivalently, if there is a non-degenerate pairing
\begin{equation*}
\beta'_{X,Y} : \mathrm{Hom}_\mathcal{C}(Y,X[d+1]) \times \mathrm{Hom}_\mathcal{C}(X,Y) \to \mathbb{K}
\end{equation*}
which is bifunctorial in $X \in \mathcal{D}$ and $Y \in \mathcal{C}$.
We will sometimes refer to $\beta'$ as the {\em relative $(d+1)$-Calabi--Yau pairing} on $(\mathcal{C},\mathcal{D})$.

Now let us recall the following notion.

\begin{dfn}[{\cite[Section 10.2]{kon-soi09}}]
A smooth $A_\infty$-category $\mathcal{A}$ is called {\em $(d+1)$-Calabi--Yau} if the inverse dualizing bimodule $\mathcal{A}_\Delta^!$ is isomorphic to $\mathcal{A}_\Delta[-d-1]$ in the derived category of $(\mathcal{A},\mathcal{A})$-bimodules.
\end{dfn}

Applying Proposition \ref{sm-prop} to a Calabi--Yau $A_\infty$-category, we obtain the following corollary.

\begin{cor}\label{CY-cor}
Let $\mathcal{A}$ be a $(d+1)$-Calabi--Yau $A_\infty$-category.
Then $(D^\pi\mathcal{A},D_c\mathcal{A})$ is relative $(d+1)$-Calabi--Yau and $D_c\mathcal{A}$ is $(d+1)$-Calabi--Yau.
\end{cor}

\section{Calabi--Yau triples}

\subsection{Calabi--Yau triples}

Let $\mathcal{C}$ be a triangulated category.
Let $\mathcal{D} \subset \mathcal{C}$ be a thick subcategory and $\mathcal{X} \subset \mathcal{C}$ be a full subcategory.
We say that $\mathcal{D}$ is {\em split-generated} by $\mathcal{X}$ (or $\mathcal{X}$ {\em split-generates} $\mathcal{D}$) if $\mathcal{D}$ coincides with the smallest thick subcategory containing $\mathcal{X}$.
An object $M \in \mathcal{C}$ is called a {\em split-generator} if $\{M\}$ split-generates $\mathcal{C}$.
A split-generator $M \in \mathcal{C}$ is called a {\em silting object} if it moreover satisfies $\mathrm{Hom}_\mathcal{C}(M,M[p])=0$ for all $p>0$.
For any object $M \in \mathcal{C}$, define
\begin{gather*}
\mathcal{C}_M^{\leq 0} \coloneqq \{ X \in \mathcal{C} \,|\, \mathrm{Hom}_\mathcal{C}(M,X[p])=0 \text{ for all } p>0 \},\\
\mathcal{C}_M^{\geq 0} \coloneqq \{ X \in \mathcal{C} \,|\, \mathrm{Hom}_\mathcal{C}(M,X[p])=0 \text{ for all } p<0 \}.
\end{gather*}

\begin{dfn}[{\cite[Definition 4.3]{amy19},\cite[Section 5.1]{iya-yan18}}]
Let $\mathcal{C}$ be a split-closed triangulated category, $M \in \mathcal{C}$ be an object and $\mathcal{D} \subset \mathcal{C}$ be a thick subcategory.
The triple $(\mathcal{C},\mathcal{D},M)$ is called an {\em ST-triple} if it satisfies the following conditions:
\begin{enumerate}
\item $M$ is a silting object and $\mathrm{Hom}_\mathcal{C}(M,X)$ is finite dimensional for every $X \in \mathcal{D}$.
\item $(\mathcal{C}_M^{\leq 0},\mathcal{C}_M^{\geq 0})$ is a t-structure on $\mathcal{C}$.
\item $\mathcal{C}_M^{\geq 0} \subset \mathcal{D}$ and $(\mathcal{D}_M^{\leq 0},\mathcal{D}_M^{\geq 0})$ is a bounded t-structure on $\mathcal{D}$.
\end{enumerate}
An ST-triple $(\mathcal{C},\mathcal{D},M)$ is called a {\em $(d+1)$-Calabi--Yau triple} ($d \geq 1$) if it further satisfies the following condition:
\begin{enumerate}
\item[(4)] $(\mathcal{C},\mathcal{D})$ is {\em relative $(d+1)$-Calabi--Yau}.
\end{enumerate}
\end{dfn}

The heart of the t-structure $(\mathcal{C}_M^{\leq 0},\mathcal{C}_M^{\geq 0})$ can be described as follows.

\begin{prop}[{\cite[Proposition 4.6]{amy19}}]\label{ST-prop}
Let $(\mathcal{C},\mathcal{D},M)$ be an ST-triple and $\mathcal{H} \coloneqq \mathcal{C}_M^{\leq 0} \cap \mathcal{C}_M^{\geq 0}$ be the heart of the t-structure $(\mathcal{C}_M^{\leq 0},\mathcal{C}_M^{\geq 0})$.
Then $\mathcal{H} = \mathcal{D}_M^{\leq 0} \cap \mathcal{D}_M^{\geq 0}$ and the functor
\begin{equation*}
\mathrm{Hom}_\mathcal{C}(M,-) : \mathcal{H} \to \mathrm{mod}_{\mathrm{End}_\mathcal{C}(M)}
\end{equation*}
is an equivalence where $\mathrm{mod}_{\mathrm{End}_\mathcal{C}(M)}$ denotes the abelian category of finitely generated right $\mathrm{End}_\mathcal{C}(M)$-modules.
\end{prop}

Since $\mathrm{End}_\mathcal{C}(M)$ is a finite dimensional algebra, the heart $\mathcal{H} \simeq \mathrm{mod}_{\mathrm{End}_\mathcal{C}(M)}$ is a length category (i.e., every object of $\mathcal{H}$ has a finite composition series) with finitely many (isomorphism classes of) simple objects.
As $(\mathcal{D}_M^{\leq 0},\mathcal{D}_M^{\geq 0})$ is a bounded t-structure on $\mathcal{D}$, this implies that the simple objects of $\mathcal{H}$ split-generate $\mathcal{D}$.

\begin{exa}[{\cite[Lemma 4.15]{amy19}, \cite[Section 2]{ami09}, \cite[Proposition 2.1]{kal-yan16}}]\label{ST-ex}
Let $R$ be a dg algebra satisfying the following conditions:
\begin{enumerate}
\item $H^p(R)=0$ for every $p>0$.
\item $H^0(R)$ is finite dimensional.
\item $D_cR \subset D^\pi R$.
\end{enumerate}
Then $(D^\pi R,D_cR,R)$ is an ST-triple.
If moreover $R$ is $(d+1)$-Calabi--Yau then $(D^\pi R,D_cR)$ is relative $(d+1)$-Calabi--Yau by Corollary \ref{CY-cor} and hence $(D^\pi R,D_cR,R)$ is a $(d+1)$-Calabi--Yau triple.
\end{exa}

\subsection{Quotient categories}

Let $\mathcal{D}$ be a triangulated subcategory of $\mathcal{C}$.
The {\em quotient category} $\mathcal{C}/\mathcal{D}$ has the same set of objects as $\mathcal{C}$ and for $X,Y \in \mathcal{C}/\mathcal{D}$, the morphism space $\mathrm{Hom}_{\mathcal{C}/\mathcal{D}}(X,Y)$ consists of equivalence classes of diagrams in $\mathcal{C}$ of the form
\begin{equation*}
s^{-1} \circ f \coloneqq
\begin{tikzcd}
X \ar[rd,"f",swap] & & Y \ar[ld,"s"]\\
& Y' &
\end{tikzcd}
\end{equation*}
such that $\mathrm{Cone}(s) \in \mathcal{D}$.
Two diagrams $s^{-1} \circ f,t^{-1} \circ g$ are equivalent if there is a commutative diagram
\begin{equation*}
\begin{tikzcd}
& Y' \ar[d] &\\
X \ar[rd,"g",swap] \ar[ru,"f"] \ar[r] & Z & Y \ar[ld,"t"] \ar[lu,"s",swap] \ar[l,"u"]\\
& Y'' \ar[u] &
\end{tikzcd}
\end{equation*}
in $\mathcal{C}$ such that $\mathrm{Cone}(u) \in \mathcal{D}$.
The composition of $s^{-1} \circ f \in \mathrm{Hom}_{\mathcal{C}/\mathcal{D}}(X,Y)$ and $t^{-1} \circ g \in \mathrm{Hom}_{\mathcal{C}/\mathcal{D}}(Y,Z)$ is defined as $(s' \circ t)^{-1} \circ (g' \circ f)$ using the diagram
\begin{equation*}
\begin{tikzcd}
X \ar[rd,"f",swap] & & Y \ar[ld,"s"] \ar[rd,"g",swap] & & Z \ar[ld,"t"]\\
& Y' \ar[rd,"g'",swap] & & Z' \ar[ld,"s'"] &\\
& & W & &
\end{tikzcd}
\end{equation*}
where the square
\begin{equation*}
\begin{tikzcd}
Y \ar[r,"s"] \ar[d,"g",swap] & Y' \ar[d,"g'"]\\
Z' \ar[r,"s'",swap] & W
\end{tikzcd}
\end{equation*}
is a homotopy cartesian square.
The quotient category $\mathcal{C}/\mathcal{D}$ has a unique triangulated structure which makes the quotient functor $Q : \mathcal{C} \to \mathcal{C}/\mathcal{D}$ exact.

\subsection{Cluster categories}

Let $\mathcal{C}$ be a triangulated category.
For an object $M \in \mathcal{C}$, we denote by $\mathrm{add}(M)$ the smallest full additive subcategory containing $M$ and closed under taking direct summands.
For two full subcategories $\mathcal{D},\mathcal{E} \subset \mathcal{C}$, we define $\mathcal{D} * \mathcal{E}$ to be the full subcategory of $\mathcal{C}$ whose set of objects is
\begin{equation*}
\{X \in \mathcal{C} \,|\, \text{there is an exact triangle } Y \to X \to Z \to Y[1] \text{ such that } Y \in \mathcal{D}, Z \in \mathcal{E} \}.
\end{equation*}
For an object $M \in \mathcal{C}$ and integers $p \leq q$, set
\begin{equation*}
M^{[p,q]} \coloneqq \mathrm{add}(M)[p] * \mathrm{add}(M)[p+1] * \cdots * \mathrm{add}(M)[q-1] * \mathrm{add}(M)[q].
\end{equation*}

Now let $\mathcal{C}$ be a $d$-Calabi--Yau triangulated category.
An object $M \in \mathcal{C}$ is called a {\em $d$-cluster-tilting object} if
\begin{equation*}
\mathrm{add}(M) = \{ X \in \mathcal{C} \,|\, \mathrm{Hom}_\mathcal{C}(M,X[p])=0 \text{ for all } 1 \leq p \leq d-1 \}.
\end{equation*}

\begin{thm}[{\cite[Theorem 5.8]{iya-yan18}}]\label{IY-thm}
Let $(\mathcal{C},\mathcal{D},M)$ be a $(d+1)$-Calabi--Yau triple.
Then we have the following:
\begin{enumerate}
\item The quotient category $\mathcal{C}/\mathcal{D}$ is a $d$-Calabi--Yau triangulated category.
\item The quotient functor $\mathcal{C} \to \mathcal{C}/\mathcal{D}$ induces an equivalence $M^{[0,d-1]} \to \mathcal{C}/\mathcal{D}$ of additive categories.
\item $M$ is a $d$-cluster-tilting object of $\mathcal{C}/\mathcal{D}$.
\end{enumerate}
\end{thm}

In this case, we will call the quotient category $\mathcal{C}/\mathcal{D}$ the {\em cluster category}.

\begin{rmk}
The $d$-Calabi--Yau pairing on $\mathcal{C}/\mathcal{D}$ can be described explicitly using the relative $(d+1)$-Calabi--Yau pairing on $(\mathcal{C},\mathcal{D})$.
For the precise construction, see \cite[Section 1]{ami09}.
\end{rmk}

For an object $M \in \mathcal{C}$, define
\begin{equation*}
\mathcal{C}^M_{\geq 0} \coloneqq \bigcup_{p \geq 0} M^{[-p,0]}, \quad \mathcal{C}^M_{\leq 0} \coloneqq \bigcup_{p \leq 0} M^{[0,p]}.
\end{equation*}

\begin{prop}[{\cite[Proposition 5.9]{iya-yan18}}]\label{IY-prop}
Let $(\mathcal{C},\mathcal{D},M)$ be a $(d+1)$-Calabi--Yau triple.
Then the quotient functor $\mathcal{C} \to \mathcal{C}/\mathcal{D}$ induces an isomorphism
\begin{equation*}
\mathrm{Hom}_\mathcal{C}(X,Y) \cong \mathrm{Hom}_{\mathcal{C}/\mathcal{D}}(X,Y)
\end{equation*}
for every $X \in \mathcal{C}^M_{\leq 0}$ and $Y \in \mathcal{C}^M_{\geq 1-d}$.
\end{prop}

Theorem \ref{IY-thm} and Proposition \ref{IY-prop} allow us to describe $\mathrm{Hom}^*_{\mathcal{C}/\mathcal{D}}(M,M)$ in terms of $\mathrm{Hom}_\mathcal{C}^*(M,M)$.

\begin{cor} \label{cor:hom}
Let $(\mathcal{C},\mathcal{D},M)$ be a $(d+1)$-Calabi--Yau triple.
Then
\begin{equation*}
\mathrm{Hom}_{\mathcal{C}/\mathcal{D}}(M,M[p]) \cong
\begin{cases}
\mathrm{Hom}_\mathcal{C}(M,M[p]) & (p \leq 0),\\
0 & (1 \leq p \leq d-1),\\
D\mathrm{Hom}_\mathcal{C}(M,M[d-p]) & (p \geq d).
\end{cases}
\end{equation*}
\end{cor}

\begin{proof}
Note that $M \in \mathcal{C}^M_{\leq 0}$ and $M[p] \in \mathcal{C}^M_{\geq 1-d} (= \mathcal{C}^M_{\geq 0}[d-1])$ for every $p \leq d-1$.
Therefore, by Proposition \ref{IY-prop}, we have $\mathrm{Hom}_{\mathcal{C}/\mathcal{D}}(M,M[p]) \cong \mathrm{Hom}_\mathcal{C}(M,M[p])$ for every $p \leq d-1$.
In particular, $\mathrm{Hom}_{\mathcal{C}/\mathcal{D}}(M,M[p])=0$ for all $1 \leq p \leq d-1$ since $M$ is a silting object of $\mathcal{C}$.
On the other hand, for $p \geq d$, we see that
\begin{equation*}
\mathrm{Hom}_{\mathcal{C}/\mathcal{D}}(M,M[p]) \cong D\mathrm{Hom}_{\mathcal{C}/\mathcal{D}}(M,M[d-p]) \cong D\mathrm{Hom}_\mathcal{C}(M,M[d-p])
\end{equation*}
where the first isomorphism follows from Theorem \ref{IY-thm} (1) and the second isomorphism follows from Proposition \ref{IY-prop}.
\end{proof}

Our main lemma is the following.
In Section \ref{section:symplectic}, we will apply the following lemma to the derived wrapped Fukaya category $\mathcal{C} = D^\pi\mathcal{W}(M)$ and the derived compact Fukaya category $\mathcal{D} = D^\pi\mathcal{F}(M)$.

\begin{lem} \label{lem:main}
Let $\mathcal{C}$ be a split-closed algebraic triangulated category and $\mathcal{D} \subset \mathcal{C}$ be a thick subcategory.
Let $M \cong M_1 \oplus \cdots \oplus M_n \in \mathcal{C}$ be a basic silting object where each $M_i$ is indecomposable.
Suppose it satisfies the following conditions:
\begin{enumerate}
\item $\mathrm{Hom}_\mathcal{C}(M,M)$ is finite dimensional.
\item There is a collection $\{S_1,\dots,S_n\} \subset \mathcal{D}$ of objects which split-generates $\mathcal{D}$ such that
\begin{equation*}
\mathrm{Hom}_\mathcal{C}(M_i,S_j[p]) \cong
\begin{cases}
\mathbb{K} & (i=j \text{ and } p=0),\\
0 & (\text{otherwise}).
\end{cases}
\end{equation*}
\end{enumerate}
Then $(\mathcal{C},\mathcal{D},M)$ is an ST-triple.
If moreover $(\mathcal{C},\mathcal{D})$ is relative $(d+1)$-Calabi--Yau then $(\mathcal{C},\mathcal{D},M)$ is a $(d+1)$-Calabi--Yau triple.
\end{lem}

\begin{proof}
Since $\mathcal{C}$ is split-closed and algebraic, we can assume without loss of generality that $\mathcal{C} = D^\pi\mathcal{A}$ for some dg category $\mathcal{A}$.
Since $M$ is a silting object, we can also assume that $\mathrm{hom}_\mathcal{A}(M,M)$ is non-positively graded.
Now consider a dg algebra $R = \mathrm{hom}_\mathcal{A}(M,M)$.
Then the dg functor $\mathrm{hom}_\mathcal{A}(M,-) : \mathcal{A} \to \mathrm{Mod}_R$ induces an exact functor $\Phi : D\mathcal{A} \to DR$ which restricts to an exact equivalence
\begin{equation*}
\mathcal{C} = D^\pi\mathcal{A} \overset{\sim}{\to} D^\pi R.
\end{equation*}

Note that the heart
\begin{equation*}
\mathcal{H} \coloneqq (D^\pi R)_R^{\leq 0} \cap (D^\pi R)_R^{\geq 0} = \{ X \in D^\pi R \,|\, H^p(X)=0 \text{ for all } p \neq 0 \}
\end{equation*}
of the standard t-structure on $D^\pi R$ is equivalent to $\mathrm{mod}_{H^0(R)}$ (see \cite[Proposition 2.1]{kal-yan16}).
By the condition (2), we see that the image of $\{S_1,\dots,S_n\}$ under $\Phi$ is contained in $\mathcal{H}$ and gives a complete set of pairwise non-isomorphic simple objects of $\mathcal{H}$.
Since $\{S_1,\dots,S_n\}$ is assumed to split-generate $\mathcal{D}$ and a complete set of pairwise non-isomorphic simple objects of $\mathcal{H}$ split-generates $D_cR$, this shows that $\Phi$ restricts to an exact equivalence
\begin{equation*}
\mathcal{D} \overset{\sim}{\to} D_cR
\end{equation*}
and thus $D_cR \subset D^\pi R$.

Consequently, since $\Phi$ sends $M$ to $R$, we see that the conditions (1), (2) and (3) in Example \ref{ST-ex} are satisfied for $R$ by the assumption on $M$ and the above argument.
Therefore we conclude that $(D^\pi R,D_cR,R)$ is an ST-triple and so is $(\mathcal{C},\mathcal{D},M)$.
\end{proof}

	\section{Examples from symplectic geometry}\label{section:symplectic}
	
	Let us briefly review the construction of the wrapped Fukaya category of a Liouville domain $M$ to fix some notations and conventions. We refer the readers to \cite{abo-sei10,abo12,gan13} for more detail.
	
	%(A pair $(M, \lambda)$ of compact smooth manifold $M$ with boundary and an one form $\lambda \in \Omega^1(M)$ on it is called a Liouville domain if $\omega = d\lambda$ is a symplectic form on $M$ and the vector field $Z$ satisfying $\iota_Z \omega = \lambda$ points outward along the boundary $\partial M$. Let us denote such a Liouville domain $(M,\lambda)$ by $M$ throughout the paper.)

	Let $(M,\omega=d\lambda)$ be a Liouville domain of dimension $2(d+1)$, which means that $(M,\omega)$ is a symplectic manifold with a boundary and its boundary $\partial M$ is a contact manifold of dimension $2d+1$ with the contact form $\lambda|_{\partial M}$. This forces $M$ to have a cylindrical end near its boundary, which is isomorphic to the symplectization $(0,1] \times \partial M$ of the boundary $\partial M$ symplectically. We will denote such a Liouville domain just by $M$ throughout the paper.
	
	We consider the completion of $M$
	$$ \widehat{M} = M \cup_{\partial M} ([1,\infty) \times \partial M)$$
	obtained from $M$ by gluing the symplectization $[1,\infty) \times \partial M$ with the identification $\partial M = \{1 \} \times \partial M$. 
	
	The {\em wrapped Fukaya category} $\mathcal{W}(M)$ of $M$ is an $A_{\infty}$-category whose objects are exact Lagrangian submanifolds $L$ of $M$ satisfying the following:
	\begin{itemize}
		\item $L$ intersects the boundary  $\partial M$ transversely, and
		\item the restriction $\lambda|_L$ vanishes near the boundary $\partial L = L \cap \partial M$.
	\end{itemize}
	We will call such Lagrangian submanifolds {\em admissible}. These requirements force any admissible Lagrangian to have a cylindrical Legendrian end near the boundary of $M$. The {\em compact Fukaya category} $\mathcal{F}(M)$ is the full subcategory of $\mathcal{W}(M)$ consisting of exact closed Lagrangian submanifolds of $M$. We will denote $\mathcal{W}(M)$ by $\mathcal{W}$ and denote $\mathcal{F}(M)$ by $\mathcal{F}$ respectively, when the Liouville domain $M$ is clear from the context.
	
	To be more precise, Lagrangian submanifolds are required to admit some additional data to be allowed as an object of $\mathcal{W}(M)$. First, in order to endow the wrapped Fukaya category $\mathcal{W}(M)$ with a $\mathbb{Z}$-grading, one needs to further assume
	$$ 2c_1(M) =0 \in H^2(M ,\mathbb{Z})$$
	and also Lagrangian submanifolds are required to be graded in the sense of \cite{sei08}. Furthermore, in order to work over a field of characteristic zero, one needs to choose a background class $b\in H^2(M,\mathbb{Z}_2)$ and require the second Stiefel--Whitney classes of the tangent bundles of Lagrangians to coincide with the restriction of $b$ \cite{abo12}. We will denote the corresponding wrapped Fukaya category and compact Fukaya category by $\mathcal{W}_b(M)$ and $\mathcal{F}_b(M)$, respectively and assume the background class $b$ to be zero unless otherwise specified.

	As we constructed the completion $\widehat{M}$ from $M$ above, for a Lagrangian submanifold $L$ with a cylindrical Legendrian end, one can consider its completion
	$$ \widehat{L} = L \cup_{\partial L} ([1,\infty) \times \partial L) \subset \widehat{M}.$$
	
	Roughly speaking, the morphism space $\hom_{\mathcal{W}} (L_0, L_1)$ between two admissible Lagrangians $L_0$ and $L_1$ of $M$ is defined by the graded $\mathbb{K}$-vector space generated by non-degenerate Hamiltonian chords from $\widehat{L}_0$ to $\widehat{L}_1$ for a Hamiltonian on $\widehat{M}$, which is quadratic in the radial coordinate of the symplectization $(0,\infty)\times \partial M$. %Refer to for more details.
	%More precisely, the morphism space $\hom_{\mathcal{W}} (L_0, L_1)$ is given by $CW(L_0, L_1)$ defined in \eqref{eq:wrappedfloercomplex1}.

	%In many examples of Lagrangians $L_1$ and $L_2$, the morphism space $\hom_{\mathcal{W}} (L_1, L_2)$ is infinite dimensional as there are infinitely many Hamiltonian chords from $L_1$ to $L_2$ for such a Hamiltonian.
	
	For every integer $k\geq 1$ and any collection of $(k+1)$ Lagrangian submanifolds $L_0, \dots ,L_k$ of $M$, the $k$-th $A_{\infty}$-product 
	$$\mu^k : \hom_{\mathcal{W}} (L_{k-1}, L_k) \otimes \dots \otimes \hom_{\mathcal{W}} (L_0, L_1) \to \hom_{\mathcal{W}} (L_0, L_k)[2-k]$$
	is defined by counting rigid pseudo-holomorphic disks with $(k+1)$ boundary punctures $\{z_i\}_{i=0}^k$, which  is asymptotic to a Hamiltonian chord from $\widehat{L}_0$ to $\widehat{L}_k$ at $z_0$, is asymptotic to a Hamiltonian chord from $\widehat{L}_{i-1}$ to $\widehat{L}_{i}$ at $z_i$ for $ 1\leq i \leq k$ and maps to $\widehat{L}_i$ on the boundary component between $z_{i}$ to $z_{i+1}$.\\
	
	%On the other hand, every Liouville manifold $M$ gives rise to a closed string invariant $SH^*(M)$ called the symplectic cohomology of $M$. Roughly speaking, it is the cohomology of the Floer cochain complex generated by Hamiltonian orbits for a quadratic Hamiltonian on $M$. It is well-known that the symplectic cohomology carries a unital ring structure defined by counting rigid pseudo-holomorphic pair-of-pants curves asymptotic to Hamiltonian orbits at its ends.  
	
	%As predicted from string theory, the Hochschild invariants of $\mathcal{W}(M)$ and $SH^*(M)$ are closely related. Indeed, there is a ring homomorphism $\mathcal{C}\mathcal{O}$, called closed-open map, from symplectic cohomology of $M$ to the Hochschild cohomology of the wrapped Fukaya category of $M$ $$ \mathcal{C}\mathcal{O} : SH^*(M) \to HH^*(\mathcal{W}(M)).$$
	%Also, regarding the Hochschild homology $HH_*(\mathcal{W}(M))$ of the wrapped Fukaya category of $M$ as a $SH^*(M)$-module with the above ring homomorphism, there is a $SH^*(M)$-module morphism $\mathcal{O}\mathcal{C}$, called open-closed map, from the Hochschild homology of the wrapped Fukaya category of $M$ to the symplectic cohomology of $M$ $$\mathcal{O}\mathcal{C} : HH_*(\mathcal{W}(M)) \to SH^*(M).$$

	%It was proved in [Abouzaid] that a set $\{L_i\}$ of Lagrangians of $M$ split-generates $\mathcal{W}(M)$ if the restriction of the open-closed map to the split-closure of $\{L_i\}$ hits the identity of $SH^*(M)$. We call a Liouville domain $M$ non-degenerate if it admits a finite collection $\{L_i\}$ of Lagrangians satisfying such a property.
	
	A Liouville domain is said to be {\em non-degenerate} if it admits a set of Lagrangians satisfying the Abouzaid's generating criterion \cite{abo10}. For example, every Weinstein manifold is a non-degenerate Liouville manifold. If $M$ is a non-degenerate Liouville domain of dimension $2(d+1)$, then $\mathcal{W}(M)$ is a smooth and $(d+1)$-Calabi--Yau $A_\infty$-category \cite{gan13}.\\

	Let us now assume that the derived wrapped Fukaya category $D^{\pi} \W(M)$ has a silting object
	\begin{equation*}%\label{eq:silting}
		L = \bigoplus_{i=1}^n L_i
	\end{equation*}
	for some connected Lagrangians $L_i$ of $M$, whose image under the Yoneda embedding is indecomposable as an object of $D^{\pi} \W(M)$.
	Let us further assume that the derived compact Fukaya category $D^{\pi} \F(M)$ is split-generated by
	$$ S = \bigoplus_{i=1}^n S_i$$
	for some compact Lagrangians $S_i$ of $M$.

	Then the following proposition is a consequence of Lemma \ref{lem:main}.
	\begin{prop}\label{prop:cytriple}
		If the following conditions hold:
		\begin{enumerate}
			\item $\Hom_{D^{\pi} \W(M)} (L,L)$ is finite dimensional.
			\item $\Hom_{D^{\pi} \W(M)} (L_i, S_j[p]) = \begin{cases}\mathbb{K} & (i = j \text{ and } p=0),\\0 & (\text{otherwise}), \end{cases}$
		\end{enumerate}
		then $(D^{\pi} \W(M), D^{\pi} \F(M), L)$ is a $(d+1)$-Calabi--Yau triple.

	\end{prop}

	\subsection{Cotangent bundle of a simply-connected manifold}\label{subsection:cotangent}
	In this subsection, we show that cotangent bundles give examples of Proposition \ref{prop:cytriple}.
	
	Indeed, let $N$ be a simply-connected closed smooth manifold of dimension $d+1$. On the one hand, we consider the wrapped Fukaya category $\W_b(T^*N)$ of the cotangent bundle $T^*N$ with the background class $b\in H^2(T^*N,\mathbb{Z}_2)$ given by the pullback of the second Stiefel--Whitney class of $N$. On the other hand, let us consider the based loop space $\Omega_x N$ for some $x\in N$. Then the concatenation of paths defines a continuous map
	$$\Omega_x N \times \Omega_x N \to \Omega_x N$$
	and it in turn gives rise to a ring structure on the homology $H_{*}(\Omega_x N)$. Furthermore, in \cite{abo12}, it was argued that the cubical chain complex of the based loop space $C_{-*} (\Omega_x N)$ carries a structure of dg algebra. Indeed its differential is given by the boundary operator and the product is induced by the concatenation of paths again.
	
	Let us consider a Lagrangian
	\begin{equation}\label{eq:siltingcotangent}
		L = T_x^* N \subset T^*N,
	\end{equation}
	the cotangent fiber of $T^*N$ fibered at the point $x \in N$. The following theorems are well-known.
	\begin{thm}[\cite{abo11}]\label{thm:A}
		The cotangent fiber $L=T^*_x N$ generates the wrapped Fukaya category $\W_b(T^*N)$.
	\end{thm}
	
	\begin{thm}[\cite{abb-sch10}, \cite{abo12}]\label{thm:ASA}
		There is an $A_{\infty}$-quasi-isomorphism
		$$ \hom_{\W_b(T^*N)} (L, L) \cong C_{-*}(\Omega_x N).$$
		In particular, there is an isomorphism of graded $\mathbb{K}$-algebra
		$$ \Hom_{D^{\pi}\W_b(T^*N)}^*(L,L) \cong H_{-*} (\Omega_x N).$$
	\end{thm}
	
	As a consequence of the above theorem, considering that $N$ is simply-connected, we have
	\begin{enumerate}
		\item
		$\dim \Hom_{\W_b(T^*N)}^0 (L, L) =1$.
		\item
		$\Hom_{D^{\pi} \W_b(T^*N)}^* (L, L)$ is non-positively graded.
	\end{enumerate}
	It follows that $L$ is a silting object of $D^{\pi}\W_b(T^*N)$.
	
	Now let us consider a compact Lagrangian
	\begin{equation}\label{eq:compactgeneratorcotangent}
		S = \text{The zero section of the cotangent bundle } T^*N \subset T^*N.
	\end{equation}
	It clearly intersects the cotangent fiber $L$ exactly once at $x \in L\cap S$ and there are no other Hamiltonian chords between $L$ and $S$ for the quadratic Hamiltonian on $T^*N$ given by the square of the fiber norm with respect to a Riemannian metric on $N$. This means that the morphism space between $L$ and $S$ is one dimensional in $\W_b(T^*N)$. Furthermore we may grade the Lagrangian $S$ so that
	%on $T^*N$ that is $C^2$-small Morse in a disk cotangent bundle in $T^*N$ and is linear with respect to the radial coordinate in the symplectization of $ST^*N =\partial T^* N$. (See the definition of an admissible Hamiltonian in Subsection \ref{subsection:preliminariesonwrappedfloertheory}). 
	$$ \Hom^p_{D^{\pi}\W_b(T^*N)} (L,S) =\begin{cases} \mathbb{K} & (p=0),\\ 0& \text{(otherwise).}\end{cases}$$
	It was shown in \cite{fss08} that every connected closed Lagrangian submanifold of $T^*N$ is quasi-isomorphic to $S$ as an object of the compact Fukaya category $\F_b(T^*N)$. This implies that the zero section $S$ generates $\F_b(T^*N)$.
	
	In summary, we just showed that $L$ and $S$ satisfy the assumption of Proposition \ref{prop:cytriple}. Hence we have
	\begin{thm}\label{thm:cotangent}
		$(D^{\pi}\W_b(T^*N),D^{\pi}\F_b(T^*N), L)$ is a $(d+1)$-Calabi--Yau triple. 
	\end{thm}

	\subsection{Plumbings of cotangent bundles $T^* S^{d+1}$}\label{subsection:plumbing}
	Generalizing the case explained in Subsection \ref{subsection:cotangent}, we prove that Proposition \ref{prop:cytriple} continues to hold for the plumbing of cotangent bundles of spheres along a tree.
	
	For that purpose, first recall that a quiver $Q$ is a directed graph, which can be defined formally as follows.
	\begin{dfn}
		A {\em quiver} $Q$ consists of a set $Q_0$ of vertices, a set $Q_1$ of arrows and a pair of maps $s$ and $t$ from $Q_1$ to $Q_0$, called a source map and a target map respectively. 
		For any $\alpha \in Q_1$ with $s(\alpha) = i \in Q_0$ and $t(\alpha) = j \in Q_0$, we also denote it by
		$$ \alpha : i \to j.$$
		
		A quiver $Q$ is said to be finite if both $Q_0$ and $Q_1$ are finite.
	\end{dfn}
	
	Let $Q$ be a finite quiver whose underlying (undirected) graph is a tree.
	For an integer $d \geq 2$, let $X^{d+1}_Q$ denote the plumbing of copies of the cotangent bundles $T^* S^{d+1}$ along the underlying tree of $Q$. Refer to \cite{etg-lek17} for a construction of such a space.
	
	%% which is obtained as follows. First take a copy $S_i$ of $(d+1)$-dimensional sphere $S^{d+1}$ for every $i\in Q_0$. Second, for each arrow $\alpha : i \to j$, choose base points $s_i \in S_i$ and $s_j \in S_j$  and glue two cotangent bundles $T^* S_i$ and $T^*S_j$ by identifying $ $

	For each $i \in Q_0$, let $S_i$ denote the zero section of the cotangent bundle $T^* S^{d+1}$ corresponding to the vertex $i$. This means that, for each arrow $\alpha :i \to j$, there exists a unique intersection point $p_{\alpha} \in S_i \cap S_j$ where the plumbing is performed. Furthermore, for each $i\in Q_0$, let $L_i$ denote the cocore disk that intersects the zero section $S_i$ transversely at exactly one point $q_i \in L_i \cap S_i$ and does not intersect any other spheres $S_j$ for $j \neq i$. Note that, for each vertex $i \in Q_0$, $S_i$ is a compact exact Lagrangian submanifold of $X^{d+1}_Q$ and hence is an object of the compact Fukaya category $\F=\F(X^{d+1}_Q)$. Also, for each vertex $i \in Q_0$, $L_i$ is an exact Lagrangian submanifold of $X^{d+1}_Q$, which is allowed to be an object of the wrapped Fukaya category $\W= \W(X^{d+1}_Q)$.

	Let us now describe certain gradings on the Lagrangians $S_i$ and $L_i$ for $i \in Q_0$.
	Note that the first Chern class $c_1(X^{d+1}_Q)$ is zero as the cohomology $H^2(X^{d+1}_Q, \mathbb{Z})$ vanishes for $d+1 \geq 3$. Furthermore both $L_i$ and $S_i$ are gradable as they are simply-connected.
	
	Since the underlying graph of $Q$ is a tree, we may grade the Lagrangian spheres $S_i$, $i \in Q_0$ in such a way that,
	for each arrow $\alpha : i \to j$, the unique intersection point $p_{\alpha} \in S_i \cap S_j$ is of degree $d$ as a morphism from $S_i$ to $S_j$. Accordingly we have
	$$ \hom_{\W} (S_i,S_j) = \hom_{\F}(S_i,S_j)= \mathbb{K} [-d].$$
	Then the $(d+1)$-Calabi--Yau property of $\W$ implies that the same intersection point $p_{\alpha} \in S_i \cap S_j$ is of degree $1$ as a morphism from $S_j$ to $S_i$ and hence we have
	$$ \hom_{\W} (S_j,S_i)  = \hom_{\F} (S_j,S_i)=  \mathbb{K} [-1].$$
	%Indeed, this can be done using inductively from a vertex of $Q$ with valency 1.
	
	Having graded the Lagrangian spheres $S_i$ for all $i\in Q_0$, we grade the cocore disks $L_i$ in such a way that the unique intersection point $q_i \in L_i\cap S_i$ is of degree $0$ as a morphism from $L_i$ to $S_i$ for each $i \in Q_0$. Hence, as in the case of the cotangent bundle, we have
	\begin{equation}\label{eq:cocoretocore}
		\mathrm{Hom}_{\W} (L_i ,S_i) = \mathbb{K} [0].
	\end{equation}
	
	With these specific gradings on $S_i$ and $L_i$, let $L_Q$ be an object of $\mathrm{Perf}_\W$ given by
	\begin{equation}\label{eq:siltingplumbing}
		L_Q = \bigoplus_{i \in Q_0} L_i.
	\end{equation}
	The result of \cite{cdrgg17} says that $L_Q$ split-generates $\mathrm{Perf}_\W$.
	Similarly let $S$ be an object of $\mathrm{Perf}_\F$ given by
	\begin{equation}\label{eq:compactgeneratorplumbing}
		S_Q= \bigoplus_{i \in Q_0} S_i.
	\end{equation}
	It is also shown in \cite{abo-smi12} that $S_Q$ split-generates $\mathrm{Perf}_\F$.

	Furthermore, it will be explained in Section \ref{sec:5} that the endomorphism algebra $\hom_{\W} (L_Q,L_Q)$ is quasi-isomorphic to the Ginzburg dg algebra $\Gamma_Q$. The result of \cite{her16} says that $\Hom_{D^{\pi}\W}(L_i,L_i)$ is one dimensional for each $i \in Q_0$, which implies that $L_i$ is indecomposable.
	It further says that $\Hom^*_{D^{\pi}\W}(L_Q,L_Q)$ is non-positively graded and that
	\begin{equation}\label{eq:endofinite}
		\dim \Hom_{D^{\pi}\W} (L_Q,L_Q)<\infty.
	\end{equation}
	
	It follows that $L_Q$ is a silting object of $D^{\pi} \mathcal{W}(X^{d+1}_Q)$. Therefore this result together with \eqref{eq:cocoretocore} and \eqref{eq:endofinite} imply that our $L_Q$ and $S_Q$ satisfy the assumption of Proposition \ref{prop:cytriple}. Hence we have
	\begin{thm}\label{thm:plubming}
		$(D^{\pi}\W(X^{d+1}_Q), D^{\pi} \F(X^{d+1}_Q), L_Q)$ is a $(d+1)$-Calabi--Yau triple.
	\end{thm}
	
	At this point, we would like to mention about a result of Ganatra, Gao and Venkatesh \cite{GGV}. They have been working on the construction of the Rabinowitz Fukaya category of a Liouville domain, which is a categorification of an open-string counterpart to the Rabinowitz Floer homology.
	\begin{thm}[\cite{GGV}]\label{thm:GGV}
		Let $M$ be a Liouville domain. If its wrapped Fukaya category $\W(M)$ admits a Koszul dual subcategory, which is a subcategory of the compact Fukaya category $\F(M)$, then the quotient category $\W(M)/\F(M)$ and the Rabinowitz Fukaya category of $M$ are quasi-equivalent.
	\end{thm}

	We have seen that, if either $M$ is the cotangent bundle of a simply-connected closed smooth manifold of dimension $d+1$ as in Subsection \ref{subsection:cotangent} or $M$ is the plumbing space $X^{d+1}_Q$ for some quiver $Q$ whose underlying graph is a tree as in Subsection \ref{subsection:plumbing}, there is an object $S$ (split-)generating the compact Fukaya category, which is given as in \eqref{eq:compactgeneratorcotangent} or \eqref{eq:compactgeneratorplumbing}. It was shown in \cite{etg-lek17} that the compact Fukaya category is a Koszul dual subcategory of the wrapped Fukaya category in either case.
	
	Moreover, in either case, we have shown that $(D^{\pi}\W, D^{\pi} \F, L)$ is a $(d+1)$-Calabi--Yau triple for some  
	silting object $L$ of the derived wrapped Fukaya category $D^{\pi} \W$ defined in \eqref{eq:siltingcotangent} or \eqref{eq:siltingplumbing}. In particular, this implies that $L$ is a generator of $D^\pi\W$ since $L$ induces a bounded co-t-structure on $D^\pi\W$ whose co-heart is given by $\mathrm{add}(L)$ (see \cite[Proposition 2.8]{iya-yan18}). Thus we have $D^\pi\W/D^\pi\F \simeq D^\pi(\W/\F)$ by \cite[Theorem 3.4]{dri04}. Combining the above observation with Theorem \ref{thm:GGV} and Lemma \ref{lem:main}, we have
	\begin{thm}\label{thm:Rabinowitz}  Under the above assumption, the following hold:
		\begin{enumerate}
			\item The derived Rabinowitz Fukaya category of $M$ is a $d$-Calabi--Yau triangulated category.
			\item $L$ is a cluster tilting object of the derived Rabinowitz Fukaya category of $M$.
		\end{enumerate}
	\end{thm}

\section{Computation of morphism spaces} \label{sec:5}

In this section, the ambient symplectic manifold $M$ is always $X_{Q}^{3}$ where $Q$ is a quiver of Dynkin type. In fact, any result in this section can be easily generalized to the case of $X_{Q}^{d+1}$ for any $d \geq 3$, but we restrict our focus to the case of $X_Q^3$ for an explicit computation (see Remark \ref{rmk:general}).

In this section, we compute the endomorphism ring
$$\mathrm{Hom}^{*}_{D^{\pi}\mathcal{W}/D^{\pi}\mathcal{F}}(L_{Q},L_{Q}) = \bigoplus_{p \in \mathbb{Z}} \mathrm{Hom}_{D^{\pi}\mathcal{W}/D^{\pi}\mathcal{F}}(L_{Q},L_{Q}[p])[-p]$$
of the cluster-tilting object $L_{Q} = \oplus L_{i}$ in $D^{\pi}\mathcal{W}/D^{\pi}\mathcal{F}$. First, we can compute the cohomology of $\hom_{\mathcal{W}}(L_{Q},L_{Q})$. By \cite{ekh-lek17}, \cite{lek-ued20}, it is shown that $\hom_{\mathcal{W}}(L_{Q},L_{Q})$ is quasi-isomorphic to the $3$-CY Ginzburg dg algebra $\Gamma_{Q}$ (for $Q$ whose underlying graph is a tree). In \cite{her16}, the minimal model of $\Gamma_{Q}$ is described.

\begin{rmk} \label{rmk:general}
Before going further, we would like to add some remarks on the general $(d+1)$-CY cases for $d \geq 3$. One missing piece for the general case is the minimal model of $\Gamma_{Q}$. In  \cite{her16}, the minimal model is given only for $3$-CY case, but the result can be generalized to the $(d+1)$-CY cases for $d\geq 3$ easily using a bigrading on $\Gamma_{Q}$. Anyway, if one can compute the cohomology of $\Gamma_{Q}$ explicitly, our results in this section can be generalized also.
\end{rmk}

\begin{dfn}
Let $Q$ be a quiver of Dynkin type. The {\em Ginzburg dg algebra} $\Gamma_{Q}$ associated to $Q$ is the path algebra of another quiver $\widehat{Q}$ which consists of the same vertex set with $Q$, i.e., $\widehat{Q}_{0} = Q_{0}$ and
\begin{itemize}
\item the arrows $Q_{1}$ of degree $0$,
\item for each $\alpha : i \to j$ in $Q_{1}$, a new arrow $\alpha^{*} : j \to i$ of degree $-1$,
\item for each vertex $i$ in $Q_{0}$, a loop $t_{i} : i \to i$ at $i$ of degree $-2$.
\end{itemize}
We denote the degree of any arrow (or path) $\gamma$ by $|\gamma|$. The differential structure is defined by
\begin{itemize}
\item for $\alpha$ in $Q_{1}$, $d \alpha = 0, d \alpha^{*} = 0$,
\item $d t_{i} = \sum (\alpha{^*} \alpha - \beta \beta^{*})$ where the sum runs over all $\alpha : i \to j$ and $\beta : k \to i$.
\end{itemize}
Here, our convention for the path $\alpha_{k} \alpha_{k-1} \dots \alpha_{2} \alpha_{1}$ is $t(\alpha_{i}) = s(\alpha_{i+1})$, $1 \leq i \leq k-1$.
For the later use, let $\overline{Q}$ denote the subquiver of $\widehat{Q}$ without the loops $t_{i}$ at each vertex $i$.
\end{dfn}

Let us define an involution $\phi$ on $Q$ for each type of Dynkin diagram. For $A_{n}$ quiver, $\phi$ is an involution defined as $\phi (i) \coloneqq n+1-i$ (regarding $i \in Q_{0}$ as an integer from $1$ to $n$). For $D_{2n+1}$ or $E_{6}$ quiver, there is the unique nontrivial involution, e.g.,
$$\begin{tikzcd}[row sep=1em, column sep=2.3em]
& & & 4 & & & & & & \phi(5) \\
1 \arrow[r] & 2 \arrow[r] & 3 \arrow[ru, "\alpha"] \arrow[rd, "\beta", swap] & & \mbox{ } \arrow[r, "\phi"] & \mbox{ } & \phi(1) \arrow[r] & \phi(2) \arrow[r] & \phi(3) \arrow[ru, "\phi(\beta)"] \arrow[rd, "\phi(\alpha)", swap] & \\
& & & 5 & & & & & & \phi(4)
\end{tikzcd}.$$
For $D_{2n}$, $E_{7}$ or $E_{8}$, we choose $\phi$ as the identity map.

\begin{dfn} \label{dfn:omega}
To define a new quiver, we add an arrow $v_{i,\phi(i)} : i \to \phi(i)$ for each $i$ to the quiver $\overline{Q}$. The degrees of such arrows are determined by an algebraic relation between Auslander--Reiten translation $\tau$ and shift functor as follow. Let $N : \overline{Q}_{0} = Q_{0} \to \mathbb{N}$ be a map which is given by
$$\tau^{-N(i)} P_{\phi(i)} = P_{i}[1].$$
Here, $P_{i}$ is the indecomposable projective module over $\mathbb{K}Q$ corresponding to $i \in Q_{0}$. Then, the degree of $v_{i,\phi_{i}}$ is defined by $-N(i)-1$.
This resulting quiver is denoted by $\Omega_{Q}$. 

Let $J$ be an ideal of the path algebra $\mathbb{K}\Omega_{Q}$ generated by
\begin{equation} \label{eqn:gen}
\begin{cases}
\mbox{for each } i \in \overline{Q}_{0}, \; \displaystyle \sum_{\substack{\alpha : i \to j \\ \beta : k \to i}} (\alpha^{*} \alpha - \beta \beta^{*}),\\
\mbox{for each } \alpha : i \to j, \; \alpha v_{\phi(i), i} - v_{\phi(j), j} \phi(\alpha).
\end{cases}
\end{equation}
\end{dfn} 

\begin{exa} \label{exa:A5}
When $Q$ is a standard $A_{5}$-quiver
$$\begin{tikzcd}[row sep=small]
1 \arrow[r] & 2 \arrow[r] & 3 \arrow[r] & 4 \arrow[r] & 5,
\end{tikzcd}$$
the quiver $\Omega_{Q}$ is given by
$$
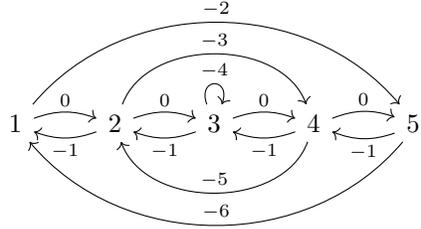
\begin{tikzcd}[row sep={4em,between origins}]
1 \arrow[r, bend left=20, "0"] \arrow[rrrr, bend left=50, "-2"] & 2 \arrow[r, bend left=20, "0"] \arrow[l, bend left=20, "-1"] \arrow[rr, bend left=70, "-3"] & 3 \arrow[r, bend left=20, "0"] \arrow[l, bend left=20, "-1"] \arrow[in=70, out=110, loop, looseness=5,"-4"] & 4 \arrow[r, bend left=20, "0"] \arrow[l, bend left=20, "-1"] \arrow[ll, bend left=70, "-5", swap] & 5. \arrow[l, bend left=20, "-1"] \arrow[llll, bend left=50, "-6", swap]
\end{tikzcd}$$
The number on each arrow denotes the degree of the corresponding arrow. We denote any arrow $\alpha : i \to j$ of degree $0$ (resp. $\alpha^{*} : j \to i$ of degree $-1$) in $\overline{Q}_{1}$ by $u_{i,j}$ (resp. $u_{j,i}$).
\end{exa}

\begin{thm}[{\cite{her16}}] \label{thm:H}
$\hom_{\mathcal{W}}(L_{Q},L_{Q})$ is quasi-isomorphic to $\mathbb{K}\Omega_{Q} / J$.
\end{thm}

Based on this theorem, we denote a morphism between the Lagrangians $L_i$ by the corresponding path in the quiver. A corollary that $\hom_{\mathcal{W}}(L_{Q},L_{Q})$ is supported in non-positive degrees follows immediately.

From now on, $Q$ denotes the standard $A_{n}$-quiver (as in Example \ref{exa:A5}). We give a specific representative for any path in $\mathbb{K}\Omega_{Q} / J$. In this case, the generators of the ideal $J$ are given by
\begin{equation} \label{eqn:rel1}
u_{12} u_{21},\; u_{32} u_{23} - u_{12} u_{21}, \; \dots, \; u_{n,n-1} u_{n-1,n} - u_{n-2,n-1} u_{n-1,n-2}, \; -u_{n-1,n} u_{n,n-1},
\end{equation}
and
\begin{equation} \label{eqn:rel2}
\begin{split}
&u_{12} v_{n,1} - v_{n-1,2} u_{n,n-1}, \; u_{23} v_{n-1,2} - v_{n-2,3} u_{n-1,n-2}, \; \dots, \; u_{n-1,n} v_{2,n-1} - v_{1,n} u_{21}, \\
&u_{21} v_{n-1,2} - v_{n,1} u_{n-1,n}, \; u_{32} v_{n-2,3} - v_{n-1,2} u_{n-2,n-1}, \; \dots, \; u_{n,n-1} v_{1,n} - v_{2,n-1} u_{12}.
\end{split}
\end{equation}

Note that any path in $\mathbb{K}\Omega_{Q} / J$ consists of $u$-arrows and $v$-arrows. Using the relations \eqref{eqn:rel2}, one can gather all the $v$-arrows in the starting point (or endpoint) of given path. For example, for the following path from $2$ to $3$ in Example \ref{exa:A5},
\begin{equation} \label{eqn:ex1}
\begin{split}
u_{23} u_{32} v_{33} u_{43} u_{54} v_{15} u_{21} &\sim u_{23} u_{32} v_{33} u_{43} u_{54} u_{45} v_{24}\\
& \sim u_{23} u_{32} u_{23} v_{42} u_{54} u_{45} v_{24} \\
& \sim u_{23} u_{32} u_{23} u_{12} v_{51} u_{45} v_{24} \\
& \sim u_{23} u_{32} u_{23} u_{12} u_{21} v_{42} v_{24}.
\end{split}
\end{equation}
Hence, any path from $i$ to $j$ can be written as
\begin{equation} \label{eqn:can}
u_{j \pm 1, j} \dots u_{i, i\pm 1} v_{\phi(i),i} \dots v_{i, \phi(i)} \mbox{ or } u_{j \pm 1, j} \dots u_{\phi(i), \phi(i)\pm 1} v_{i,\phi(i)} v_{\phi(i),i} \dots v_{i, \phi(i)}.
\end{equation}
We call this a canonical representation. For the later, we give a list of useful relations coming from the original relations. The proof is an easy exercise.

\begin{lem} \label{lem:eqn}
The following equations hold in $\mathbb{K}\Omega_{Q} / J$.
\begin{enumerate}
\item $u_{i,i\pm1} v_{\phi(i),i} v_{i,\phi(i)} = v_{\phi(i)\mp1,i\pm1} v_{i\pm1,\phi(i)\mp1} u_{i,i\pm1}$.
\item $u_{i\pm1,i} u_{i,i\pm1} v_{\phi(i),i} v_{i,\phi(i)} = v_{\phi(i),i} v_{i,\phi(i)} u_{i\pm1,i} u_{i,i\pm1}$.
\item $u_{i\pm1,i} (u_{i\pm2,i\pm1} u_{i\pm1,i\pm2})^{k} u_{i,i\pm1} = (u_{i\pm1,i} u_{i,i\pm1})^{k+1}$.
\end{enumerate}
\end{lem}

There are two types of $u$-arrows; one is an increasing type $u_{i,i+1}$ and the other is a decreasing type $u_{i,i-1}$. The relations \eqref{eqn:rel1} give ways to change the order of increasing type $u$-arrow and decreasing type $u$-arrow. Then we can prove the following lemma.

\begin{lem} \label{lem:u}
Let $P$ be a path from $i$ to $j$ only consisting of $u$-arrows. Denote the number of increasing type $u$-arrows in $P$ by $I(P)$ and the number of decreasing type $u$-arrows in $P$ by $D(P)$. Then,
$$I(P) - D(P) = j-i.$$
Moreover, $P$ becomes zero in $\mathbb{K}\Omega_{Q}/J$ if and only if $I(P) > \min \{n-i, j-1\}$ holds.

\begin{proof}
The first equation is directly given by the definitions of $I(P)$ and $D(P)$. Assume that $i<j$, $\min \{n-i, j-1\} = n-i$ without loss of generality and $I(P) > n-i$. Using the relations \eqref{eqn:rel1}, one can arrange $P$ into
$$u_{i+1,i} \dots u_{n,n-1} u_{n-1,n} \dots u_{j,j+1} u_{j-1,j} \dots u_{i+1,i+2} u_{i,i+1}.$$
The first $n-i$ $u$-arrows are of the increasing type. Since $I(P) > n-i$, there should be an arrow $u_{n,n-1}$ in $P$ after $n-i$ $u$-arrows and after that, $u_{n-1,n}$ should come. Hence, $P$ vanishes in $\mathbb{K}\Omega_{Q}/J$ because its arranged form contains the subpath $u_{n-1,n} u_{n,n-1}$. All the other cases can be proved in a similar way.
Also, by the same reason, if $P$ contains $u_{n-1,n} u_{n,n-1}$ or $u_{21} u_{12}$, then $I(P)$ should be greater than $\min \{n-i, j-1\}$.
\end{proof}
\end{lem}

As an example, let $P$ be $u_{23} u_{32} u_{23} u_{12} u_{21}$ in \eqref{eqn:ex1}. By definition, $I(P) = 3$, $D(P) = 2$ and $3 = I(P) > \min \{ 3, 2 \} = 2$ hold. Then,
\begin{equation*} \label{eqn:ex2}
u_{23} u_{32} u_{23} u_{12} u_{21} \sim u_{23} u_{12} u_{21} u_{12} u_{21}
\end{equation*}
and hence it contains $u_{21} u_{12} \in J$.

Now, we compute the endomorphism ring $\mathrm{Hom}^{*}_{D^{\pi}\mathcal{W}}(L_{Q},L_{Q})$ using path algebra with relations $\mathbb{K}\Omega_{Q} / J$ for standard $A_{n}$-quiver $Q$. First, $\mathrm{Hom}^{*}_{D^{\pi}\mathcal{W}}(L_{i},L_{i})$ consists of all paths from $i$ to $i$ in $\mathbb{K}\Omega_{Q} / J$. Any such path can be written as
\begin{equation*}
u_{i \pm 1, i} \dots u_{i, i\pm 1} v_{\phi(i),i} \dots v_{i, \phi(i)} \mbox{ or } u_{i \pm 1, i} \dots u_{\phi(i), \phi(i)\pm 1} v_{i,\phi(i)} v_{\phi(i),i} \dots v_{i, \phi(i)}.
\end{equation*}
We decompose these paths into three parts $X$, $Y$ and $Z$. $X$ part is given by
$$\begin{cases}
u_{i+1,i} \dots u_{\phi(i),\phi(i)-1} v_{i, \phi(i)} = v_{\phi(i),i} u_{\phi(i)-1,\phi(i)} \dots u_{i,i+1} &\mbox{ if } i<\phi(i),\\
u_{i-1,i} \dots u_{\phi(i),\phi(i)+1} v_{i, \phi(i)} = v_{\phi(i),i} u_{\phi(i)+1,\phi(i)} \dots u_{i,i-1} &\mbox{ if } i>\phi(i),\\
v_{i, \phi(i)} &\mbox{ if } i=\phi(i).
\end{cases}$$
$Y$ part is a path $(v_{\phi(i),i} v_{i, \phi(i)})^{l}$ for some $l \geq 0$ from $i$ to $i$ which consists of only $v$-arrows. $Z$ part is a path from $i$ to $i$ which consists of only $u$-arrows and therefore satisfies $I(Z) = D(Z)$. With these notations, the path is decomposed into $ZY$ if the number of $v$-arrows is even or $ZXY$ if the number of $v$-arrows is odd.

Let $x_{i}$ be the subpath corresponding to $X$ part, $y_{i} = v_{\phi(i),i} v_{i,\phi(i)}$ and $z_{i} = u_{i+1,i} u_{i, i+1} = u_{i-1,i} u_{i,i-1}$. Their degrees are given by $|x_{i}| = \min\{n-i,i-1\} -n-1 $, $|y_{i}| = -n-3$ and $|z_{i}| = -1$. Then, we have the following isomorphisms.

\begin{prop} \label{prop:ieqj}
There is a ring isomorphism
$$\mathrm{Hom}^{*}_{D^{\pi}\mathcal{W}}(L_{i},L_{i}) \cong
\begin{cases}
\mathbb{K}[x_{i},y_{i},z_{i}]/\langle z_{i}^{k_{i}}, x_{i}^{2} - y_{i} \rangle \cong \mathbb{K}[x_{i},z_{i}]/\langle z_{i}^{k_{i}} \rangle &(\mbox{if }n\mbox{ is odd and }i = \frac{n+1}{2}), \\
\mathbb{K}[x_{i},y_{i},z_{i}]/\langle z_{i}^{k_{i}}, x_{i}^{2} - y_{i}z_{i}^{l_{i}} \rangle &(\mbox{otherwise})
\end{cases}$$
where $k_{i} = \min \{ n-i, i-1 \} + 1$ and $l_{i} = |\phi(i) - i|$.

\begin{proof}
First, we need to show that the subpaths $x_{i}$, $y_{i}$ and $z_{i}$ commute in $\mathbb{K}\Omega_{Q} / J$. For example, the following modification when $i<\phi(i)$
\begin{align*}
u_{i+1,i} \dots u_{\phi(i),\phi(i)-1} v_{i, \phi(i)} v_{\phi(i),i} v_{i,\phi(i)} &\sim v_{\phi(i),i} u_{\phi(i)-1,\phi(i)} \dots u_{i,i+1} v_{\phi(i),i} v_{i,\phi(i)} \\
& \sim v_{\phi(i),i} u_{\phi(i)-1,\phi(i)} \dots u_{i+1,i+2} v_{\phi(i)-1,i+1} u_{\phi(i),\phi(i)-1} v_{i,\phi(i)} \\
& \hspace{10em} \vdots \\
& \sim v_{\phi(i),i} v_{i,\phi(i)} u_{i+1,i} \dots u_{\phi(i),\phi(i)-1} v_{i,\phi(i)}
\end{align*}
gives $x_{i}y_{i} = y_{i}x_{i}$. One can check $y_{i}z_{i}=z_{i}y_{i}$ and $x_{i}z_{i}=z_{i}y_{i}$ in a similar way using Lemma \ref{lem:eqn}. By Lemma \ref{lem:u}, since $I(z_{i}^{k_{i}}) = k_{i} = \min \{ n-i, i-1 \} + 1$, $z_{i}^{k_{i}}$ vanishes and $z_{i}^{k}$ does not vanish for $k < k_{i}$. Next, $x_{i}^{2}$ can be written as
$$u_{i+1,i} \dots u_{\phi(i),\phi(i)-1} v_{i, \phi(i)} v_{\phi(i),i} u_{\phi(i)-1,\phi(i)} \dots u_{i,i+1}$$
if $i<\phi(i)$ and by Lemma \ref{lem:eqn} (a), it becomes $(u_{i+1,i} u_{i,i+1})^{l_{i}} v_{\phi(i),i} v_{i,\phi(i)}$. When $i=\phi(i)$ or $i>\phi(i)$, all the computations are similar.

From the canonical representative of the path in $\Omega_{Q}$, one can see that any relation containing more than two $v$-arrows should be in the ideal $\langle z_{i}^{k_{i}}, x_{i}^{2} - y_{i}z_{i}^{l_{i}} \rangle$, i.e., only these two are the minimal relations. Thus, there is an injective ring homomorphism
$$f : \mathbb{K}[x_{i},y_{i},z_{i}]/\langle z_{i}^{k_{i}}, x_{i}^{2} - y_{i}z_{i}^{l_{i}} \rangle \to \mathrm{Hom}^{*}_{D^{\pi}\mathcal{W}}(L_{i},L_{i}).$$
Also from the canonical representative, any path can be written as
$$z_{i}^{a}y_{i}^{b} \mbox{ or } z_{i}^{a}x_{i}y_{i}^{b}$$
for some $a$, $b \in \mathbb{Z}_{\geq 0}$. Hence, $f$ is a ring isomorphism.
\end{proof}
\end{prop}

Let $U_{ij}$ be the shortest path consisting of $u$-arrows from $i$ to $j$ such that $I(U_{ij})=0$ if $i>j$ or $D(U_{ij}) =0$ if $i<j$. We define $U_{ii}=e_{i}$ if $i=j$ for the consistent notation. Denote the path $U_{\phi(i),j} v_{i,\phi(i)}$ by $V_{ij}$. Then, one can see the generation result of the morphism space from their canonical representative for the case of $i \neq j$.

\begin{prop} \label{prop:ineqj}
Suppose $i \neq j$. Then two paths $U_{ij}$ and $V_{ij}$ generate the morphism space $\mathrm{Hom}^{*}_{D^{\pi}\mathcal{W}}(L_{i},L_{j})$ as a $\mathrm{Hom}^{*}_{D^{\pi}\mathcal{W}}(L_{j},L_{j})$-module.
\end{prop}

\begin{rmk}
There is a relation between $U_{ij}$ and $V_{ij}$. There are $4$ possible cases when $i<j$.
\begin{itemize}
\item $i<j \leq \phi(j)<\phi(i)$ : $V_{ij} = x_{j} U_{ij}$,
\item $i \leq \phi(j) \leq j \leq \phi(i)$ : $z_{j}^{j-\phi(j)}V_{ij} = x_{j}U_{ij}$,
\item $\phi(j)<\phi(i) \leq i<j$ : $z_{j}^{j-i}V_{ij} = x_{j}U_{ij}$,
\item $\phi(j) \leq i \leq \phi(i) \leq j$ : $z_{j}^{j-i}V_{ij} = x_{j}U_{ij}$.
\end{itemize}
Note that if $j-\phi(j) \geq k_{j}$ or $j-i \geq k_{j}$, then the corresponding terms in the left hand side are zero in fact.
There are also $4$ cases when $i>j$ and one can get a similar list by exchanging $i$ and $j$.
\end{rmk}

From Propositions \ref{prop:ieqj} and \ref{prop:ineqj}, we can compute any component of $\mathrm{Hom}^{*}_{D^{\pi}\mathcal{W}}(L_{Q},L_{Q})$ explicitly. Note that $\mathrm{Hom}^{>0}_{D^{\pi}\mathcal{W}}(L_{Q},L_{Q}) = 0$ since the positive degree part of $\mathbb{K}\Omega_{Q} / J$ is trivial. We move on the quotient category $D^{\pi}\mathcal{W}/D^{\pi}\mathcal{F}$. Corollary \ref{cor:hom} gives a relation between the morphism spaces in $D^{\pi}\mathcal{W}$ and $D^{\pi}\mathcal{W}/D^{\pi}\mathcal{F}$.

\begin{cor}
As a vector space, 
$$\mathrm{Hom}_{D^{\pi}\mathcal{W}/D^{\pi}\mathcal{F}}(L_{Q},L_{Q}[p]) \cong \begin{cases}
\mathrm{Hom}_{D^{\pi}\mathcal{W}}(L_{Q},L_{Q}[p]) & (p \leq 0), \\
0 & (p = 1), \\
D\mathrm{Hom}_{D^{\pi}\mathcal{W}}(L_{Q},L_{Q}[2-p]) & (p \geq 2).
\end{cases}$$
\end{cor}

We investigate the ring structure on $\mathrm{Hom}^{*}_{D^{\pi}\mathcal{W}/D^{\pi}\mathcal{F}}(L_{Q},L_{Q})$. The product structure on the non-positive part $\mathrm{Hom}^{\leq 0}_{D^{\pi}\mathcal{W}/D^{\pi}\mathcal{F}}(L_{Q},L_{Q})$ is the same as that on $\mathrm{Hom}^{*}_{D^{\pi}\mathcal{W}}(L_{Q},L_{Q})$. We need to see the product with positive degree morphisms.

By the definition, we can find some positive degree morphisms in $\mathrm{Hom}^{*}_{D^{\pi}\mathcal{W}/D^{\pi}\mathcal{F}}(L_{Q},L_{Q})$. Note that since $\mathrm{Hom}^{*}_{D^{\pi}\mathcal{W}}(L_{Q},L_{Q})$ is supported in non-positive degrees, any positive degree morphism should be a newly added invertible morphism. The following lemma gives the examples of those new morphisms.

\begin{lem} \label{lem:cone}
Let $v_{i,\phi(i)}$ be the homogeneous element of $\mathrm{Hom}^{*}_{D^{\pi}\mathcal{W}}(L_{i},L_{\phi(i)})$. Then, the mapping cone of $v_{i,\phi(i)}$ belongs to $D^{\pi}\mathcal{F}$ for any $i$.

\begin{proof}
Consider the mapping cone $\mathrm{Cone}(v_{i,\phi_{i}})$ of $v_{i,\phi(i)}$ in $D^{\pi}\mathcal{W}$:
\begin{equation} \label{eqn:iso}
\hom_{\mathcal{W}}(L_{Q},L_{i}) \xrightarrow{\hspace{0.3em} \mu^{2}(v_{i,\phi(i)},-) \hspace{0.3em}} \hom_{\mathcal{W}}(L_{Q},L_{\phi(i)}).
\end{equation}
From the relations \eqref{eqn:rel2}, one can easily see that $\ker v_{i,\phi(i)} = 0$, i.e., concatenating with $v_{i,\phi(i)}$ is injective. The image consists of all paths containing at least one $v$-arrow. Hence, the cokernel can be represented by the paths which consist of only the $u$-arrows. Since such paths have the maximal length by Lemma \ref{lem:u}, there are finitely many such paths. Thus, the complex \eqref{eqn:iso} is in $D_{c}{\mathcal{W}}$.
\end{proof}
\end{lem}

The above lemma implies that the morphism $v_{i,\phi(i)}$ in $\mathrm{Hom}^{*}_{D^{\pi}\mathcal{W}/D^{\pi}\mathcal{F}}(L_{i},L_{\phi(i)})$ has a formal inverse $v_{i,\phi(i)}^{-1}$ in $\mathrm{Hom}^{*}_{D^{\pi}\mathcal{W}/D^{\pi}\mathcal{F}}(L_{\phi(i)},L_{i})$ for any $1 \leq i \leq n$. We will show that they are enough to determine the positively graded part of $\mathrm{Hom}^{*}_{D^{\pi}\mathcal{W}/D^{\pi}\mathcal{F}}(L_{Q},L_{Q})$. To do that, we introduce a modification of the quiver $\Omega_{Q}$.

\begin{dfn} \label{dfn:newomega}
	Let $\overline{\Omega}_{Q}$ be a quiver obtained from $\Omega_{Q}$ by adding an arrow $v_{i,\phi(i)}^{-1} : \phi(i) \to i$ for each $i \in \Omega_{Q,0}$. The ideal $\overline{J}$ in $\mathbb{K} \overline{\Omega}_{Q}$ is generated by all the generators in \eqref{eqn:gen} (they are also elements of $\mathbb{K} \overline{\Omega}_{Q}$) and
$$v_{i,\phi(i)}^{-1}v_{i,\phi(i)} - e_{i},\;v_{i,\phi(i)}v_{i,\phi(i)}^{-1} - e_{\phi(i)}$$
for all $i \in \Omega_{Q,0}$.
\end{dfn}

Now, we can state the main theorem of this section.

\begin{thm} \label{thm:main}
There is a graded ring isomorphism $\phi : \mathbb{K}\overline{\Omega}_{Q}/\overline{J} \xrightarrow{\hspace{0.3em} \simeq \hspace{0.3em}} \mathrm{Hom}^{*}_{D^{\pi}\mathcal{W}/D^{\pi}\mathcal{F}}(L_{Q},L_{Q})$.

\begin{proof}
For the non-positively graded part, $\phi$ is defined in the same way as the isomorphism in Theorem \ref{thm:H}. The image of the arrow $v_{i,\phi(i)}^{-1}$ is a morphism in $\mathrm{Hom}^{*}_{D^{\pi}\mathcal{W}/D^{\pi}\mathcal{F}}(L_{\phi(i)},L_{i})$ given by
$$\begin{tikzcd}[row sep={4em,between origins}, column sep={5em,between origins}]
L_{\phi(i)} \ar[rd,"\mathrm{id}",swap] & & L_{i} \ar[ld,"v_{i,\phi(i)}"]\\
& L_{\phi(i)} &
\end{tikzcd}.$$
Then, we need to show that $\phi$ is well-defined. The only relations that we need to check are $v_{i,\phi(i)}^{-1}v_{i,\phi(i)} - e_{i}$, $v_{i,\phi(i)}v_{i,\phi(i)}^{-1} - e_{\phi(i)}$. The compositions of $\phi(v_{i,\phi(i)})$ and $\phi(v_{i,\phi(i)}^{-1})$ are given by
$$\begin{tikzcd}[row sep={4em,between origins}, column sep={5em,between origins}]
L_{i} \ar[rd,"v_{i,\phi(i)}",swap] & & L_{\phi(i)} \ar[ld,"\mathrm{id}"] \ar[rd,"\mathrm{id}",swap] & & L_{i} \ar[ld,"v_{i,\phi(i)}"]\\
& L_{\phi(i)} \ar[rd,"\mathrm{id}",swap] & & L_{\phi(i)} \ar[ld,"\mathrm{id}"] &\\
& & L_{\phi(i)} & &
\end{tikzcd} \mbox{and}
\begin{tikzcd}[row sep={4em,between origins}, column sep={5em,between origins}]
L_{\phi(i)} \ar[rd,"\mathrm{id}",swap] & & L_{i} \ar[ld,"v_{i,\phi(i)}",swap] \ar[rd,"v_{i,\phi(i)}"] & & L_{\phi(i)} \ar[ld,"\mathrm{id}"]\\
& L_{\phi(i)} \ar[rd,"\mathrm{id}",swap] & & L_{\phi(i)} \ar[ld,"\mathrm{id}"] &\\
& & L_{\phi(i)} & &
\end{tikzcd}.$$
The first one is equivalent to the identity by the following diagram
$$\begin{tikzcd}[row sep={4em,between origins}, column sep={8em,between origins}]
& L_{\phi(i)} \ar[d,"\mathrm{id}"] &\\
L_{i} \ar[rd,"\mathrm{id}",swap] \ar[ru,"v_{i,\phi(i)}"] \ar[r,"v_{i,\phi(i)}"] & L_{\phi(i)} & L_{i} \ar[ld,"\mathrm{id}"] \ar[lu,"v_{i,\phi(i)}",swap] \ar[l,"v_{i,\phi(i)}",swap]\\
& L_{i} \ar[u,"v_{i,\phi(i)}"] &
\end{tikzcd}.$$
Hence, $\phi$ is a well-defined ring homomorphism by the construction.

Next, we denote the degree $p$ piece of $\mathbb{K} \overline{\Omega}_{Q}/\overline{J}$ by $(\mathbb{K} \overline{\Omega}_{Q}/\overline{J})_{p}$. Then, we will show that
\begin{equation} \label{eqn:dim}
\dim (\mathbb{K} \overline{\Omega}_{Q}/\overline{J})_{p} = \dim (\mathbb{K} \overline{\Omega}_{Q}/\overline{J})_{2-p}
\end{equation}
 for $p \leq 0$. Note that each graded piece is finite dimensional. Recall that any path of negative degree from $i$ to $j$ can be written as
$$u_{j\pm1,j} \dots u_{j\pm k, j\pm(k-1)} u_{j\pm(k-1), j\pm k} \dots u_{i,i\pm1} v_{\phi(i),i} \dots v_{i,\phi(i)} = U_{j\pm k,j} U_{i,j\pm k} (v_{\phi(i),i} v_{i,\phi(i)})^{l}$$
if the number of $v$-arrows is even or
$$u_{j\pm1,j} \dots u_{j\pm k, j\pm(k-1)} u_{j\pm(k-1), j\pm k} \dots u_{\phi(i),\phi(i)\pm1} v_{i,\phi(i)} \dots v_{i,\phi(i)} = U_{j\pm k,j} U_{\phi(i),j\pm k} v_{i,\phi(i)} (v_{\phi(i),i} v_{i,\phi(i)})^{l}$$
if the number of $v$-arrows is odd for some $k$, $l \in \mathbb{Z}_{\geq 0}$. For any path from $i$ to $j$ of positive degree, its canonical representation is given by
$$u_{j\pm1,j} \dots u_{j\pm k, j\pm(k-1)} u_{j\pm(k-1), j\pm k} \dots u_{i,i\pm1} v_{i,\phi(i)}^{-1} \dots v_{\phi(i),i}^{-1} = U_{j\pm k,j} U_{i,j\pm k} (v_{i,\phi(i)}^{-1} v_{\phi(i),i}^{-1})^{l}$$
if the number of $v$-arrows is even, or
$$u_{j\pm1,j} \dots u_{j\pm k, j\pm(k-1)} u_{j\pm(k-1), j\pm k} \dots u_{\phi(i),\phi(i)\pm1} v_{\phi(i),i}^{-1} \dots v_{\phi(i),i}^{-1} = U_{j\pm k,j} U_{\phi(i),j\pm k} v_{\phi(i),i}^{-1} (v_{i,\phi(i)}^{-1} v_{\phi(i),i}^{-1})^{l}$$
if the number of $v$-arrows is odd for some $k$, $l \in \mathbb{Z}_{\geq 0}$.

We consider the case $i<j \leq \phi(j)<\phi(i)$ only. Note that we divide the cases because of the notational problem. In any other cases $\phi(j) \leq i \leq \phi(i) \leq j$, $\phi(i)<\phi(j) \leq j<i$, etc., there are only small modifications according to their order. Let $U_{j+k,j} U_{i,j+k} (v_{\phi(i),i} v_{i,\phi(i)})^{l}$ be a path of degree $p \leq 0$ from $i$ to $j$. By Lemma \ref{lem:u}, the inequality $j+k-i \leq j-1$, i.e., $k \leq i-1$ holds. Then there is always a path of degree $2-p$ from $j$ to $i$ which is given by
\begin{equation} \label{eqn:dual}
U_{i-k',i} U_{\phi(j),i-k'} v_{\phi(j),j}^{-1} (v_{j,\phi(j)}^{-1} v_{\phi(j),j}^{-1})^{l}
\end{equation}
for $k'$ such that $k + k' = i-1$. One can check the degree of this path by a direct computation. The concatenation of those two is
\begin{equation} \label{eqn:deg2}
\begin{split}
&U_{j+k,j} U_{i,j+k} (v_{\phi(i),i} v_{i,\phi(i)})^{l} U_{i-k',i} U_{\phi(j),i-k'} v_{\phi(j),j}^{-1} (v_{j,\phi(j)}^{-1} v_{\phi(j),j}^{-1})^{l} \\
&= U_{j+k,j} U_{i,j+k} U_{i-k',i} U_{\phi(j),i-k'} v_{\phi(j),j}^{-1} \\
&= U_{j+k,j} U_{i-k',j+k} U_{\phi(j),i-k'} v_{\phi(j),j}^{-1} \\
&= U_{i-k-k',j} U_{\phi(j),i-k-k'} v_{\phi(j),j}^{-1} \\
&= U_{1,j} U_{\phi(j),1} v_{\phi(j),j}^{-1} \\
&= U_{1,j} v_{1,n}^{-1} U_{j,n}
\end{split}
\end{equation}
Since $U_{1,j}$ and $U_{j,n}$ are of degree $0$, the degree of $U_{1,j} v_{1,n}^{-1} U_{j,n}$ is always $2$ for any $j$. Thus, the path given in \eqref{eqn:dual} is of degree $2-p$. Similarly, if we start with $U_{j-k,j} U_{\phi(i),j-k} v_{i,\phi(i)} (v_{\phi(i),i} v_{i,\phi(i)})^{l}$ of degree $p$, then we can find a path
$$U_{i-k',i} U_{j,i-k'} (v_{j,\phi(j)}^{-1} v_{\phi(j),j}^{-1})^{l+1}$$
such that $k + k' = i-1$. One can get the degree of this path in a similar way. This proves the equation $\dim (\mathbb{K} \overline{\Omega}_{Q}/\overline{J})_{p} = \dim (\mathbb{K} \overline{\Omega}_{Q}/\overline{J})_{2-p}$.

Finally, we prove the injectivity on each positive graded piece. Suppose that
$$\phi(U_{j\pm k,j} U_{i,j\pm k} (v_{i,\phi(i)}^{-1} v_{\phi(i),i}^{-1})^{l}) = 0$$
for some $l > 0$. Since $\phi$ is a graded ring homomorphism, we have
$$\phi(U_{j\pm k,j} U_{i,j\pm k}) = \phi(U_{j\pm k,j} U_{i,j\pm k} (v_{i,\phi(i)}^{-1} v_{\phi(i),i}^{-1})^{l}) \phi((v_{i,\phi(i)}v_{\phi(i),i})^{l})=0.$$
However, the degree of $U_{j\pm k,j} U_{i,j\pm k}$ is non-positive and $\phi$ is an isomorphism on that part. Thus, it gives a contradiction and $\phi$ is injective. We can conclude that $\phi$ is a graded ring isomorphism from \eqref{eqn:dim} and the injectivity.
\end{proof}
\end{thm}

Even though we consider the $A_{n}$ case, the same strategy can be applied to the $D$ cases and $E$ cases. By Definition \ref{dfn:omega} and relations \eqref{eqn:gen}, any path can be written as the canonical representation \eqref{eqn:can}. For $v$-arrows, Lemma \ref{lem:cone} holds also in these cases. Then, we add the formal inverses of $v$-arrows and get a new quiver as in Definition \ref{dfn:newomega}. Therefore, by the same computation, we can generalize the main theorem to any $ADE$ case. The only difference is that the notation becomes complicated since there is a trivalent vertex in $D$ and $E$ cases.

\begin{thm} \label{thm:main2}
For any Dynkin quiver $Q$, there is a graded ring isomorphism
$$\phi : \mathbb{K}\overline{\Omega}_{Q}/\overline{J} \xrightarrow{\hspace{0.3em} \simeq \hspace{0.3em}} \mathrm{Hom}^{*}_{D^{\pi}\mathcal{W}/D^{\pi}\mathcal{F}}(L_{Q},L_{Q}).$$
\end{thm}

\begin{exa}
Let $Q$ be the following $E_{6}$-quiver
$$\begin{tikzcd}
& & 2 \arrow[d] & & \\
1 \arrow[r] & 3 \arrow[r] & 4 \arrow[r] & 5 \arrow[r] & 6. 
\end{tikzcd}$$
The quiver $\Omega_{Q}$ is given by
$$
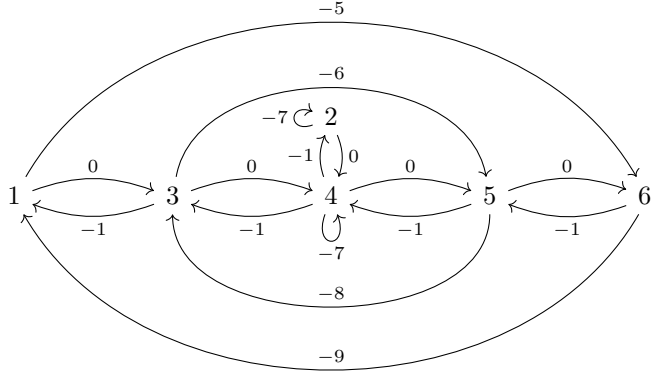
\begin{tikzcd}[column sep={6em,between origins}, row sep={3em,between origins}]
& & 2 \arrow[d, bend left=20, "0"] \arrow[in=160, out=200, loop, looseness=5,"-7"]& & \\
1 \arrow[r, bend left=20, "0"] \arrow[rrrr, bend left=60, "-5"] & 3 \arrow[r, bend left=20, "0"] \arrow[l, bend left=20, "-1"] \arrow[rr, bend left=80, "-6"] & 4 \arrow[r, bend left=20, "0"] \arrow[l, bend left=20, "-1"] \arrow[u, bend left=20, "-1"] \arrow[in=-70, out=-110, loop, looseness=8,"-7",swap] & 5 \arrow[r, bend left=20, "0"] \arrow[l, bend left=20, "-1"] \arrow[ll, bend left=90, "-8", swap] & 6. \arrow[l, bend left=20, "-1"] \arrow[llll, bend left=60, "-9", swap]
\end{tikzcd}$$
The quiver $\overline{\Omega}_{Q}$ is obtained by adding the formal inverses $v_{i,\phi(i)}^{-1}$. Hence, we can compute $\mathrm{Hom}^{*}_{D^{\pi}\mathcal{W}/D^{\pi}\mathcal{F}}(L_{Q},L_{Q})$. We give the computational results for each $L_{i}$.
\begin{enumerate}
\item $\mathrm{Hom}^{*}_{D^{\pi}\mathcal{W}/D^{\pi}\mathcal{F}}(L_{1},L_{1}) \cong \mathbb{K}[x_{1},y_{1}^{\pm1}]/\langle x_{1}^{2} \rangle$, $|x_{1}| = -9$, $|y_{1}| = -14$.
\item $\mathrm{Hom}^{*}_{D^{\pi}\mathcal{W}/D^{\pi}\mathcal{F}}(L_{2},L_{2}) \cong \mathbb{K}[x_{2}^{\pm1}]$, $|x_{2}| = -7$.
\item $\mathrm{Hom}^{*}_{D^{\pi}\mathcal{W}/D^{\pi}\mathcal{F}}(L_{3},L_{3}) \cong \mathbb{K}[x_{3},y_{3}^{\pm1},z_{3}]/\langle z_{3}^{2}, x_{3}^{2} \rangle$, $|x_{3}| = -8$, $|y_{3}| = -14$, $|z_{3}| = -1$.
\item $\mathrm{Hom}^{*}_{D^{\pi}\mathcal{W}/D^{\pi}\mathcal{F}}(L_{4},L_{4}) \cong \mathbb{K} \langle x_{4}^{\pm1},z_{4},w_{4} \rangle / \langle x_{4}z_{4} - z_{4}x_{4}, x_{4}w_{4} - w_{4}x_{4}, z_{4}^{3}, w_{4}^{3}, (z_{4}-w_{4})^{2} \rangle$, $|x_{4}| = -7$, $|z_{4}| = -1$, $|w_{4}| = -1$.
\item $\mathrm{Hom}^{*}_{D^{\pi}\mathcal{W}/D^{\pi}\mathcal{F}}(L_{5},L_{5}) \cong \mathbb{K}[x_{5},y_{5}^{\pm1},z_{5}]/\langle z_{5}^{2}, x_{5}^{2} \rangle$, $|x_{5}| = -8$, $|y_{5}| = -14$, $|z_{5}| = -1$.
\item $\mathrm{Hom}^{*}_{D^{\pi}\mathcal{W}/D^{\pi}\mathcal{F}}(L_{6},L_{6}) \cong \mathbb{K}[x_{6},y_{6}^{\pm1}]/\langle x_{6}^{2} \rangle$, $|x_{1}| = -9$, $|y_{1}| = -14$.
\end{enumerate}
Note that at the vertex $4$, the relation of $u$-arrows is given by
$$u_{54}u_{45} - u_{34}u_{43} - u_{24}u_{42}$$
and we set $z_{4} = u_{54}u_{45}$ and $w_{4} = u_{34}u_{43}$ which do not commute with each other.

For the cases of $i \neq j$, one can get a similar description as in Proposition \ref{prop:ineqj}.
\end{exa}

We close this section with a few remarks.

\begin{rmk}
%From [EL], $\mathcal{W}(X_{Q}^{d+1})$ has a Koszul dual subcategory of compact Lagrangians $\mathcal{F}(X_{Q}^{d+1})$ in the sense of [Efi]. In forthcoming work [GGV], they prove that the Rabinowitz Fukaya category $\mathcal{RW}(X_{Q}^{d+1})$ is equivalent to the quotient $\mathcal{W}(X_{Q}^{d+1})/\mathcal{F}(X_{Q}^{d+1})$. our results implies that $\mathcal{RW}(X_{Q}^{d+1})$ has a $d$-cluster category structure.
By Theorem \ref{thm:GGV}, the morphism spaces $\mathrm{Hom}^{*}_{D^{\pi}\mathcal{W}/D^{\pi}\mathcal{F}}(L_i,L_j)$ in this section 
are isomorphic to the Rabinowitz Floer cohomology of cocores $L_i$ and $L_j$ in $X^3_{Q}$.
\end{rmk}

\begin{cor}
Let $L_{i}$ and $L_{j}$ be two Lagrangian cocores in $X_{Q}^{d+1}$. The Rabinowitz Floer cohomology from $L_{i}$ to $L_{j}$ is isomorphic to $e_{j} \cdot \mathbb{K} \overline{\Omega}_{Q}/\overline{J} \cdot e_{i}$ as a graded vector space.
\end{cor}

\begin{rmk}
In \cite{bae-kwo21}, the authors computed the v-shaped Floer homology of some real Lagrangians in the Milnor fibers of Brieskorn--Pham singularities. Especially, for the $A_{n}$ case, two cocores $L_{1}$ and $L_{n}$ are realized by some real Lagrangians and we have the following isomorphism (over $\mathbb{Z}_{2}$ coefficient)
$$\widecheck{\mathrm{HW}}^{*}(L_{1},L_{1}) \cong \widecheck{\mathrm{HW}}^{*}(L_{n},L_{n}) \cong \mathbb{Z}_{2}[x,y,y^{-1}]/\langle x^{2} \rangle \cong \mathrm{Hom}^{*}_{D^{\pi}\mathcal{W}/D^{\pi}\mathcal{F}}(L_{1},L_{1}) \cong \mathrm{Hom}^{*}_{D^{\pi}\mathcal{W}/D^{\pi}\mathcal{F}}(L_{n},L_{n}).$$
\end{rmk}

\begin{rmk}
Under the isomorphisms, there is a non-degenerate pairing $\beta' : \mathbb{K} \overline{\Omega}_{Q}/\overline{J} \times \mathbb{K} \overline{\Omega}_{Q}/\overline{J} \to \mathbb{K}$ on $\mathbb{K} \overline{\Omega}_{Q}/\overline{J}$. We can decompose them according to the idempotent decomposition
$$\beta'_{ij} : e_{i} \cdot (\mathbb{K} \overline{\Omega}_{Q}/\overline{J})_{2-p} \cdot e_{j} \otimes e_{j} \cdot (\mathbb{K} \overline{\Omega}_{Q}/\overline{J})_{p} \cdot e_{i} \to \mathbb{K}.$$
Let $P_{1}$ be a path of degree $p$ from $i$ to $j$ and $P_{2}$ be a path of degree $2-p$ from $j$ to $i$. Then, $\beta'_{ij}(P_{2},P_{1})$ is defined by the coefficient of $U_{1,j} v_{1,n}^{-1} U_{j,n}$ which appears in the proof of Theorem \ref{thm:main}. It describes the non-degenerate pairing on $D^{\pi}\mathcal{W}/D^{\pi}\mathcal{F}$ introduced in \cite{ami09} (up to a nonzero scalar).
\end{rmk}

\bibliographystyle{amsalpha}
\bibliography{Bibliography}

\newcommand{\etalchar}[1]{$^{#1}$}
\providecommand{\bysame}{\leavevmode\hbox to3em{\hrulefill}\thinspace}
\providecommand{\MR}{\relax\ifhmode\unskip\space\fi MR }
\providecommand{\MRhref}[2]{%
  \href{http://www.ams.org/mathscinet-getitem?mr=#1}{#2}
}
\providecommand{\href}[2]{#2}
\begin{thebibliography}{CDRGG17}

\bibitem[Abo10]{abo10}
M.~Abouzaid, \emph{A geometric criterion for generating the {F}ukaya category},
  Publ. Math. Inst. Hautes {\'E}tudes Sci. \textbf{112} (2010), 191--240.

\bibitem[Abo11]{abo11}
\bysame, \emph{A cotangent fibre generates the {F}ukaya category}, Adv. Math.
  \textbf{228} (2011), 894--939.

\bibitem[Abo12]{abo12}
\bysame, \emph{On the wrapped {F}ukaya category and based loops}, J. Symp.
  Geom. \textbf{10} (2012), 27--79.

\bibitem[Ami09]{ami09}
C.~Amiot, \emph{Cluster categories for algebras of global dimension 2 and
  quivers with potential}, Ann. Inst. Fourier \textbf{59} (2009), 2525--2590.

\bibitem[AMY19]{amy19}
T.~Adachi, Y.~Mizuno, and D.~Yang, \emph{Discreteness of silting objects and
  t-structures in triangulated categories}, Proc. Lond. Math. Soc. \textbf{118}
  (2019), 1--42.

\bibitem[AS10a]{abb-sch10}
A.~Abbondandolo and M.~Schwarz, \emph{Floer homology of cotangent bundles and
  the loop product}, Geom. Topol. \textbf{14} (2010), 1569--1722.

\bibitem[AS10b]{abo-sei10}
M.~Abouzaid and P.~Seidel, \emph{An open string analogue of {V}iterbo
  functoriality}, Geom. Topol. \textbf{14} (2010), 627--718.

\bibitem[AS12]{abo-smi12}
M.~Abouzaid and I.~Smith, \emph{Exact {L}agrangians in plumbings}, Geom. Funct.
  Anal. \textbf{22} (2012), 785--831.

\bibitem[BJK23]{BJK23}
H.~Bae, W.~Jeong, and J.~Kim, \emph{Calabi--{Y}au structures on {R}abinowitz
  {F}ukaya categories}, arXiv:2304.02561 (2023).

\bibitem[BK21]{bae-kwo21}
H.~Bae and M.~Kwon, \emph{A computation of the ring structure in wrapped
  {F}loer homology}, Math. Z. \textbf{299} (2021), 1155--1196.

\bibitem[BMR{\etalchar{+}}06]{bmrrt06}
A.B. Buan, R.~Marsh, M.~Reineke, I.~Reiten, and G.~Todorov, \emph{Tilting
  theory and cluster combinatorics}, Adv. Math. \textbf{204} (2006), 572--618.

\bibitem[CDRGG17]{cdrgg17}
B.~Chantraine, G.~Dimitroglou~Rizell, P.~Ghiggini, and R.~Golovko,
  \emph{Geometric generation of the wrapped {F}ukaya category of {W}einstein
  manifolds and sectors}, arXiv:1712.09126 (2017).

\bibitem[CHO20]{cho20}
K.~Cieliebak, N.~Hingston, and A.~Oancea, \emph{Poincar{\'e} duality for loop
  spaces}, arXiv:2008.13161 (2020).

\bibitem[Dri04]{dri04}
V.~Drinfeld, \emph{D{G} quotients of {DG} categories}, J. Algebra \textbf{272}
  (2004), 643--691.

\bibitem[EL17a]{ekh-lek17}
T.~Ekholm and Y.~Lekili, \emph{Duality between {L}agrangian and {L}egendrian
  invariants}, arXiv:1701.01284 (2017).

\bibitem[EL17b]{etg-lek17}
T.~Etg{\"u} and Y.~Lekili, \emph{Koszul duality patterns in {F}loer theory},
  Geom. Topol. \textbf{21} (2017), 3313--3389.

\bibitem[FSS08]{fss08}
K.~Fukaya, P.~Seidel, and I.~Smith, \emph{The symplectic geometry of cotangent
  bundles from a categorical viewpoint}, Homological mirror symmetry, Lecture
  Notes in Phys., vol. 757, Springer, Berlin, 2008, pp.~1--26.

\bibitem[Gan13]{gan13}
S.~Ganatra, \emph{Symplectic cohomology and duality for the wrapped {F}ukaya
  category}, Ph.D. thesis, MIT, 2013.

\bibitem[Gan21]{gan21}
\bysame, \emph{Categorical non-properness in wrapped {F}loer theory},
  arXiv:2104.06516 (2021).

\bibitem[GGV22]{GGV}
S.~Ganatra, Y.~Gao, and S.~Venkatesh, \emph{Rabinowitz {F}ukaya categories and
  the categorical formal punctured neighborhood of infinity}, arXiv:2212.14863
  (2022).

\bibitem[Her16]{her16}
S.~Hermes, \emph{Minimal model of {G}inzburg algebras}, J. Algebra \textbf{459}
  (2016), 389--436.

\bibitem[IY18]{iya-yan18}
O.~Iyama and D.~Yang, \emph{Silting reduction and {C}alabi--{Y}au reduction of
  triangulated categories}, Trans. Amer. Math. Soc. \textbf{370} (2018),
  7861--7898.

\bibitem[Kel08]{kel08}
B.~Keller, \emph{Calabi--{Y}au triangulated categories}, Trends in
  representation theory of algebras and related topics, EMS Ser. Congr. Rep.,
  Eur. Math. Soc., Z{\"u}rich, 2008, pp.~467--489.

\bibitem[Kel11]{kel11}
\bysame, \emph{Deformed {C}alabi--{Y}au completions}, J. Reine Angew. Math.
  \textbf{654} (2011), 125--180.

\bibitem[KS09]{kon-soi09}
M.~Kontsevich and Y.~Soibelman, \emph{Notes on {$A_\infty$}-algebras,
  {$A_\infty$}-categories and non-commutative geometry}, Homological mirror
  symmetry, Lecture Notes in Phys., vol. 757, Springer, Berlin, 2009,
  pp.~153--219.

\bibitem[KY16]{kal-yan16}
M.~Kalck and D.~Yang, \emph{Relative singularity categories {I}: {A}uslander
  resolutions}, Adv. Math. \textbf{301} (2016), 973--1021.

\bibitem[LU18]{lek-ued18}
Y.~Lekili and K.~Ueda, \emph{Homological mirror symmetry for {M}ilnor fibers
  via moduli of {$A_\infty$}-structures}, arXiv:1806.04345 (2018).

\bibitem[LU20]{lek-ued20}
\bysame, \emph{Homological mirror symmetry for {M}ilnor fibers of simple
  singularities}, arXiv:2004.07374 (2020).

\bibitem[Lun10]{lun10}
V.A. Lunts, \emph{Categorical resolution of singularities}, J. Algebra
  \textbf{323} (2010), 2977--3003.

\bibitem[Orl16]{orl16}
D.~Orlov, \emph{Smooth and proper noncommutative schemes and gluing of {DG}
  categories}, Adv. Math. \textbf{302} (2016), 59--105.

\bibitem[Sei08]{sei08}
P.~Seidel, \emph{Fukaya categories and {P}icard--{L}efschetz theory}, Zurich
  Lectures in Advanced Mathamatics, European Mathematical Society, Z{\"u}rich,
  2008.

\end{thebibliography}

\end{document}